\documentclass[11pt]{amsart}
\usepackage{mathdots}
\usepackage{hyperref}

\usepackage{url}
\usepackage{amsmath,amsthm,amssymb,mathtools,tocvsec2,pdflscape,cite}
\usepackage{enumerate}
\theoremstyle{plain}
\newtheorem{thm}{Theorem}[section]
\newtheorem{prop}[thm]{Proposition}
\newtheorem{cor}[thm]{Corollary}
\numberwithin{equation}{section}
\newtheorem{lemma}[equation]{Lemma}
\newtheorem{theorem}[equation]{Theorem}
\newtheorem{utheorem}{\textrm{\textbf{Theorem}}}

\theoremstyle{definition}

\newtheorem{defn}[thm]{Definition}

\newtheorem{remark}[thm]{Remark}
\newcommand{\A}{\mathrm{P}}
\newcommand{\B}{\mathrm{Q}}
\newcommand{\1}{\mathbf{1}}
\newcommand{\W}{\mathcal{W}}
\newcommand{\diag}{\mathrm{diag}}

\newcommand{\C}{\mathbb{C}}

\newcommand{\Hil}{\mathcal{H}}
\newcommand{\Bou}{\mathcal{B}}

\newcommand{\conj}{\mathcal{C}}
\newcommand{\prd}{\mathbf{Prod_+}}

\newcommand{\jury}{\diamond}
\DeclareMathOperator{\Ran}{Ran}
\DeclareMathOperator{\Tr}{Tr}
\DeclareMathOperator{\tr}{tr}

\DeclareMathOperator{\rk}{rk}

\DeclareMathOperator{\ro}{\boldsymbol{\Theta}}

\title[Sharp lower bounds for generalized operator products]{Sharp lower bounds for generalized operator products}

\author[D.~Guillot, J.~Mashreghi, P.K.~Vishwakarma]{Dominique Guillot, Javad Mashreghi,\\ and Prateek Kumar Vishwakarma}

\address[Dominique Guillot]{ Department of Mathematical Sciences, University of Delaware, Newark, DE, USA 19716.}
\email{dguillot@udel.edu}

\address[Javad Mashreghi]{D\'epartement de math\'ematiques et de statistique, Universit\'e Laval, Qu\'ebec, QC, Canada G1V 0K6.}
\email{javad.mashreghi@ulaval.ca}

\address[Prateek Kumar Vishwakarma]{D\'epartement de math\'ematiques et de statistique, Universit\'e Laval, Qu\'ebec, QC, Canada G1V 0K6.}
\email{prateek-kumar.vishwakarma.1@ulaval.ca,~prateekv@alum.iisc.ac.in}

\date{\today}


\begin{document}

\begin{abstract}
We consider general bilinear products defined by positive semidefinite matrices. Typically non-commutative, non-associative, and non-unital, these products preserve positivity and include the classical Hadamard, Kronecker, and convolutional products as special cases. We prove that every such product satisfies a sharp nonzero lower bound in the Loewner order, generalizing previous results of Vybíral [{\em Adv. Math.}, 2020] and Khare [{\em Proc. Amer. Math. Soc.}, 2021] that were obtained in the special case of the Hadamard product. Our results naturally extend to Hilbert spaces for a family of products parametrized by positive trace-class operators, providing a lower bound in the Loewner order for such general products, including for the Hilbert tensor product.
\end{abstract}

\maketitle


\section{Introduction}\label{S:intro}

\subsection{The Loewner order}
The Loewner order, introduced by K.~Loewner, is the canonical partial order on self-adjoint operators and a fundamental tool in matrix analysis and operator theory~\cite{Loewner1934}. The Loewner order $\succcurlyeq$ compares self-adjoint operators $A,B$ on a Hilbert space $\mathcal{H}$ via
\[
A \succcurlyeq B
\quad\Longleftrightarrow\quad
\langle Ax,x\rangle \ge \langle Bx,x\rangle
\qquad \forall\, x \in \mathcal{H}.
\]
Equivalently, $A \succcurlyeq B$ if and only if $A-B$ is positive semidefinite. The order compares quadratic forms and is particularly well suited to variational arguments, spectral comparison, and operator inequalities. In finite dimensions it coincides with the usual order on Hermitian matrices, while in infinite dimensions it remains compatible with the spectral theorem and is closed under norm operator limits~\cite{conway2000course,kadison1997fundamentals}. The Loewner order plays a central role in the formulation of operator inequalities and has important applications in optimization, probability, statistics, and quantum theory. See, e.g., the monographs~\cite{bhatia2009positive, khare2022matrix, zhan2004matrix} and the references therein.

A key feature of the Loewner order is its interaction with functional calculus: Loewner's theorem characterizes operator-monotone functions as precisely those whose functional calculus preserves the Loewner order~\cite{Loewner1934,donoghue2012monotone}. In particular, the function $t^\alpha$ is operator monotone for $0\le \alpha\le 1$, and $\log t$ is operator monotone, while inversion on the positive definite cone is order reversing with respect to the Loewner order. These results connect analytic structure with order-theoretic behavior and provide a systematic source of Loewner order preserving maps. This perspective also motivates the study of positivity-preserving maps beyond functional calculus, arising from algebraic constructions on matrices and operators. Classical examples include entrywise positivity preservers, characterizing functions or kernels whose entrywise action preserves positive semidefiniteness, and preservers induced by other structured bilinear operations. See~\cite{pascoe2019noncommutative,belton2019panorama,convolution-pres,guillot2025positivity} and the references therein for recent developments and applications.

Prominent instances of positivity preserving transformations include the Hadamard product and the Kronecker product: the Schur Product Theorem asserts that the Hadamard product of positive semidefinite matrices is again positive semidefinite~\cite{schur1911}, while the Kronecker product preserves positivity in a natural tensorial sense. Both operations satisfy rich inequalities in the Loewner order (see, e.g.,~\cite{bhatia2009positive,BhatiaMatrix}). More recently, attention has also turned to convolution-type products~\cite{agler2002pick, jurythesis}, which interpolate between entrywise and tensorial constructions and yield further positivity-preserving operations~\cite{convolution-pres}.

\subsection{Matrix products preserving positivity}\label{S:matrix-products}

\noindent For an integer $n\geq 1$ and matrices $A=(a_{ij})_{i,j=1}^{n}$ and $B=(b_{ij})_{i,j=1}^{n} \in \C^{n\times n}$, consider the following three classical matrix products.
\begin{align*}
\mbox{The Hadamard (or the entrywise) product:}&\qquad A\circ B \in \C^{n\times n} \\
\mbox{The matrix convolution (or the Jury product):}&\qquad A\jury B\in \C^{n\times n}\\
\mbox{The (standard) Kronecker product:}&\qquad A\otimes B\in \C^{n^2\times n^2}
\end{align*} 
The matrix entries of these products are given by:
\begin{align*}
(A\circ B)_{ij} &:= a_{ij}b_{ij} \\
(A\jury B)_{ij} &:= \sum_{k=1}^{i}\sum_{l=1}^{j} a_{kl}b_{i-k+1,j-l+1} \\
(A\otimes B)_{ij} &:= a_{ij} B \mbox{ is the $(i,j)$-th block of $A\otimes B$}.   
\end{align*}
It is well known that these products preserve positivity. 

\begin{theorem}\label{T:schur+jury} Fix an integer $n\geq 1$, and let $A,B\in \C^{n\times n}$. Then:
\begin{eqnarray*}
A \succcurlyeq 0 \quad \mbox{and} \quad B \succcurlyeq 0\quad  \implies \quad 
\begin{cases}
A\circ B \succcurlyeq 0 & \quad \mbox{\emph{(Schur \cite{schur1911})}} \\
A\jury B \succcurlyeq 0 & \quad \mbox{\emph{(Jury \cite{jurythesis,agler2002pick})}} \\
A\otimes B \succcurlyeq 0 & \quad \mbox{\emph{(Folklore)}}
\end{cases}
\end{eqnarray*}    
\end{theorem}

In the case of the Hadamard product, a stronger form of the above was recently obtained via a conjecture of Novak, a proof by Vyb\'iral, and a refinement by Khare. The result is as follows:

\begin{theorem}[Novak's Conjecture \cite{novak1999intractability}, Vyb\'iral \cite{vybiral2020variant}, Khare \cite{khare2021sharp}]\label{khare2021sharp}
Fix integers $\beta,n\geq 1$ and nonzero matrices $A,B\in \C^{n\times \beta}$. Then 
\begin{equation*}
    AA^* \circ BB^* \; \succcurlyeq \; \frac{1}{\min{(\rk AA^*, \rk BB^*)}} \; d_{AB^{T}} d_{AB^T}^{*}  \;\succcurlyeq\; 0, 
\end{equation*}
where, given a square matrix $M=(m_{jk})$, we denote by $d_{M}:= (m_{jj})$ the column vector containing the diagonal entries of $M$. Moreover, the coefficient $1/\min{(\cdot,\cdot)}$ is best possible.
\end{theorem}

Notice that Theorem~\ref{khare2021sharp} strengthens Theorem~\ref{T:schur+jury}, albeit only for the Hadamard product~$\circ$. More precisely, it shows that the Hadamard product of positive semidefinite matrices is not merely positive semidefinite (as asserted in Theorem~\ref{T:schur+jury} for~$\circ$), but in fact enjoys a lower bound of rank at most~1. This naturally prompts analogous questions for other matrix products, most notably the classical Kronecker product~$\otimes$ and the more recent matrix convolution~$\jury$. 

\subsection{An overview of our main contributions}\label{S:overview}

Motivated by the aforementioned classical and contemporary considerations, one of our goals is to establish counterparts of Theorem~\ref{khare2021sharp} for both the Kronecker and the convolutional matrix product. More generally, we carry this out in a unified framework for generalized bilinear matrix products, and also extend our approach to products of trace-class operators on Hilbert spaces.\medskip

Before introducing our general framework, we state our main results in three special settings: for the Kronecker product of matrices, for Jury's convolutional product of matrices, and for the Hilbert tensor product of trace-class operators.

\begin{theorem}[Kronecker product]\label{T:kronecker-ext}
Let $m,n,\beta\geq 1$ be integers. Then for matrices $A\in \C^{m\times \beta}$ and $B\in \C^{n\times \beta}$, we have
\begin{align*}
AA^* \otimes BB^* &\;\succcurlyeq\; 
\frac{1}{\min(\rk AA^*,\, \rk BB^*)}\; \mathrm{vec}(BA^T)\,{\mathrm{vec}(BA^T)}^* \;\succcurlyeq\; 0,
\end{align*}
where $\mathrm{vec}(P)\in \C^{mn}$ is formed by stacking the columns of $P\in \C^{n\times m}$ in a column vector. Moreover, the coefficient \(1/\min(\rk AA^*,\, \rk BB^*)\) is the best possible, in the sense that it cannot be improved over all $A,B$.
\end{theorem}

An analogous refinement holds for convolution~$\jury$ as well.

\begin{theorem}[Matrix convolution]\label{T:jury-ext}
Let $N,\beta\geq 1$ be integers. Then for matrices $A,B\in \C^{N\times \beta}$ we have
\begin{align*}
AA^* \jury BB^* \;&\succcurlyeq\; 
\frac{1}{\min(\rk AA^*,\, \rk BB^*)}\; \rho_{\jury}\,\rho_{\jury}^* \;\succcurlyeq\; 0,\\
\mbox{where}\qquad \rho_{\jury}\;&:=\;\begin{pmatrix}\mathcal{S}_1(BA^T)&\cdots&\mathcal{S}_N(BA^T)\end{pmatrix}^T\;\in \; \C^N,
\end{align*}
with $\mathcal{S}_k(-)$ as the sum of the entries on $k$-th ``anti-diagonal'' given by
\begin{align*}
\mathcal{S}_k(P):=\sum_{i=1}^{k}p_{i,k-i+1} \qquad \forall P=(p_{ij})_{i,j=1}^{N},~k=1,\dots,N.
\end{align*}
Again, the coefficient \(1/\min(\rk AA^*,\, \rk BB^*)\) is the best possible.
\end{theorem}

We present an analogous extension for the canonical tensor product in the Hilbert space setting. While it is well known that the Hilbert tensor product $\otimes$ of positive operators is again positive, we show that this positivity can be strengthened by establishing a uniform, sharp, nonzero, lower bound in the Loewner order for trace-class operators -- see Section~\ref{S:Hilbert-prelim} for the notation.

\begin{theorem}[Hilbert tensor product]\label{T:Hilbert-tensor}
Let $H_1,H_2,\widetilde{\Hil}$ be Hilbert spaces. Then for all nonzero $A\in \Bou_2(\widetilde{\Hil},H_1)$ and $B\in \Bou_2(\widetilde{\Hil}^{\#},H_2)$,
\begin{align*}
    AA^* \otimes BB^* \; \succcurlyeq \; \frac{1}{\min(\rk AA^*, \rk BB^*)} \ro_{\mathrm{vec}(BA^\#),\mathrm{vec}(BA^\#)} \; \succcurlyeq \; 0,
\end{align*}
where we follow the convention $\frac{1}{\infty}:=0$. Moreover, ${1}/{\min(\rk AA^*, \rk BB^*)}$ is the best possible scalar, i.e., it cannot be improved uniformly over $A,B$.
\end{theorem}

\noindent As we will show, the above are all manifestations of a general construction for bilinear products parametrized by positive operators in the following main results:
\begin{description}
    \item[Theorem~\ref{T:Schur-ext}] Over Euclidean spaces for products parameterized by positive semidefinite matrices. See Page~\pageref{T:Schur-ext}.
    \item[Theorem~\ref{T:Schur-ext-Hil}] Over Hilbert spaces for products parameterized by positive trace-class operators. See Page \pageref{T:Schur-ext-Hil}.
\end{description}


From a unifying perspective, the classical Hadamard product and the Schur Product Theorem, and the standard Kronecker product and its semigroupoid structure, have played foundational roles in matrix analysis. The bilinear products in Theorem~\ref{T:Schur-ext} admit explicit formulas that are directly linked to these classical constructions. This concreteness yields uniform, dimension-dependent lower bounds in the Loewner order, which have no direct counterpart in the abstract theory. In the Hilbert space setting in Theorem~\ref{T:Schur-ext-Hil}, the same phenomena must be formulated in terms of operator ideals, spectral decompositions, conjugate spaces, and canonical isometries. The proofs rely on nontrivial issues of convergence and admissibility that are absent in finite dimension.

By presenting a unified and strengthened extension across general bilinear matrix- and operator-valued products, our work opens new avenues for theory and applications in matrix analysis, operator theory, functional analysis, combinatorial matrix theory, and related areas concerned with products and positivity.

\subsection{Organization of the paper} The rest of the paper is organized as follows. Section \ref{S:Euclidean-Intro} introduces our general framework and main results for matrices defined over Euclidean space. Section \ref{S:Euclidean} contains the proofs of our results over Euclidean spaces. Then in Section \ref{S:Hilbert} we show how our main results can be extended to positive bilinear products parameterized by the trace-class operators on Hilbert spaces, followed by the proofs in Section~\ref{S:Hilbert-proofs}. In Section~\ref{S:canonical-form} we present the canonical formulation of our bilinear products.

\section{Main results over Euclidean spaces}\label{S:Euclidean-Intro}

Henceforth, we endow each $\C^{k_1\times k_2}$, for $k_1,k_2\geq 1,$ with the standard Hilbert--Schmidt inner product given by $\langle A, B \rangle := \tr(AB^*)$ for $A,B\in \C^{k_1\times k_2}$. Here is the recipe that defines our main objects of focus.

\begin{defn}\label{defn:prd}
Fix integers $m,n,N\geq 1$.
\begin{enumerate}[(1)]
    \item Any bilinear product $\star:\C^{m\times m}\times \C^{n\times n} \longrightarrow  \C^{N\times N}$ is parameterized by a matrix $\mathcal{Y}=\mathcal{Y}(\star):=(Y_{ij})_{i,j=1}^{N}\in \C^{mnN\times mnN}$, where $Y_{ij}\in \C^{mn\times mn}$. More precisely,
\begin{align*}
    A\star B:=\begin{pmatrix}\langle A\otimes B,Y_{ij}\rangle \end{pmatrix}_{i,j=1}^{N} =
    \begin{pmatrix}\tr((A\otimes B) Y_{ij}^*)\end{pmatrix}_{i,j=1}^{N},
\end{align*}
where $A\otimes B$ is the standard Kronecker product.

\item Among all the bilinear $\star$-type products, we choose the ones parametrized by a positive semidefinite matrix:
\begin{align*}
&\prd(m,n;N):=\\
& \big\{\star:\C^{m\times m}\times \C^{n\times n} \longrightarrow \C^{N\times N} ~\big| ~\mathcal{Y}=\mathcal{Y}(\star) \mbox{ is positive semidefinite}\big\}.
\end{align*}
\end{enumerate}
When the context is clear, we write $\prd$ for $\prd(m,n,;N)$.
\end{defn}

The classical Hadamard and Kronecker product, as well as Jury's convolution product, are all special cases of the above construction.

\begin{theorem}[$\circ,\jury,\otimes \in \prd$]\label{T:schur,jury,kronecker}
For integers $m,n,N\geq 1$:
\begin{enumerate}
\item The Hadamard (entrywise) product $\circ \in \prd(N,N;N)$.
\item The matrix convolution $\jury\in \prd(N,N;N)$.
\item The Kronecker product $\otimes\in \prd(m,n;mn)$.
\end{enumerate}
In fact, each $\star\in \{\circ,\jury,\otimes\}$ is parameterized by a rank-one positive semidefinite $\mathcal{Y}=\mathcal{Y}(\star)$ of a proper size.
\end{theorem}

The set $\prd$ is naturally endowed with an addition and scalar multiplication. For integers $m,n,N\geq 1$, products $\star_1,\star_2\in \prd(m,n;N)$, and $\alpha,\beta\geq 0$, we define $\alpha\star_1+\beta\star_2$ via:
\begin{align*}
    A(\alpha\star_1+\beta\star_2)B:= \alpha A\star_1 B +\beta A\star_2B \qquad \forall A\in \C^{m\times m},B\in \C^{n\times n}.
\end{align*}
Also, for a sequence $(\star_k)_{k\geq 1}$, the limit $\star:=\lim_{_{k\to \infty}} \star_k$ is defined by the entrywise limit of matrices $\mathcal{Y}(\star):=\lim_{k\to \infty} \mathcal{Y}(\star_k)$, if it exists. Observe that, from the above definition, we have
\[
\mathcal{Y}(\alpha \star_1 + \beta \star_2) = \alpha \mathcal{Y}(\star_1) + \beta \mathcal{Y}(\star_2)
\]
and 
\[
\mathcal{Y}(\lim_{k \to \infty} \star_k) = \lim_{k \to \infty} \mathcal{Y}(\star_k) 
\]
if the limits exist. Under the above operations, the set $\prd(m,n; N)$ forms a closed convex cone. 

\begin{prop}[$\prd$ is a closed convex cone]\label{T:prod-convex}
Fix integers $m,n,N\geq 1$, and equip $\prd=\prd(m,n;N)$ with the addition and scalar multiplication as above. Then:
\begin{enumerate}[$(a)$]
    \item Provided it exists, the limit $\lim_{_{k\to \infty}} \star_k \in \prd$ for all $(\star_k)_{k\geq 1}\subset \prd$.
    \item The combination $\alpha\star_1+\beta\star_2 \in \prd$, for all $\star_1,\star_2 \in \prd$, and all $\alpha,\beta\geq 0$.
    \item Moreover, for each $\star \in \prd$, there exist $\star_1,\dots,\star_\alpha\in \prd$ such that each $\mathcal{Y}_i:=\mathcal{Y}_i(\star_i)$ is rank-one, and $\star=\star_1+\dots+\star_\alpha$.
\end{enumerate}
\end{prop}

In particular, the collection of matrix products $\prd$ forms a convex cone generated by those $\star \in \prd$ arising from rank-one positive semidefinite $\mathcal{Y}=\mathcal{Y}(\star)$.

\begin{defn}[Rank of a product $\star \in \prd$]
We say that $\star \in \prd$ has rank $\alpha$ if $\rk\left(\mathcal{Y}(\star)\right) = \alpha$, where $\rk X$ denotes the rank of the matrix $X$.
\end{defn}

As an application, our next main result yields extensions of the classical Schur Product Theorem~\cite{schur1911}, as well as its recent stronger version conjectured by Novak~\cite{novak1999intractability}, proved by Vyb\'iral~\cite{vybiral2020variant}, and subsequently sharpened by Khare~\cite{khare2021sharp}.  In particular, we obtain an uncountable family of analogues of the Schur Product Theorem and its optimal lower-bound refinement. This leads to what may be viewed as a \emph{stronger Schur Product Theorem} valid for every $\star \in \prd$.

\begin{utheorem}\label{T:Schur-ext}
Let $m,n,N,\alpha\geq 1$ be integers. For $\star \in \prd(m,n;N)$ of rank $\alpha\geq 1$, the following holds:
\begin{enumerate}[$(a)$]
\item For all $P\in \C^{m\times m}$ and $Q\in \C^{n\times n}$:
\[
P \; \succcurlyeq \; 0 \; \mbox{ and }\; Q \; \succcurlyeq \; 0 \quad \implies \quad P\star Q \; \succcurlyeq \; 0.
\]
\item More strongly, for all nonzero $A \in \C^{m\times \beta}$ and $B\in \C^{n\times \beta}$ for a given integer $\beta\geq 1$,
\[
AA^* \star BB^* \; \succcurlyeq \;  \sum_{k=1}^{\alpha}\frac{1}{\min\big(\rk AA^*,\rk BB^*,{\bf r}(k)\big)} \rho_k\, \rho_k^* \; \succcurlyeq \; 0,
\]
where the definition of each $\rho_k$ and ${\bf r}(k)$ is given next. 
\end{enumerate}  
From Proposition~\ref{T:prod-convex} we know that $\star=\sum_{j=1}^{\alpha}\star_j$ where each $\mathcal{Y}_k=\mathcal{Y}_k(\star_k)$ is rank-one positive semidefinite.
Thus, every $\mathcal{Y}_{k}=\begin{pmatrix}v_{i,k}v_{j,k}^*\end{pmatrix}_{i,j=1}^{N}$ for some $v_{1,k},\dots,v_{N,k}\in \C^{mn}$. Suppose $\mathrm{vec}(M)\in \C^{mn}$ is formed by stacking the columns of $M\in \C^{n\times m}$ into a column vector. Then $\rho_k$ is defined by:
\begin{align*}
\rho_k\;=\;\rho_{\star_k}\;:=\;
\begin{pmatrix}\,\langle\,\mathrm{vec}(BA^T),\,v_{1,k}\rangle & \cdots & \langle \mathrm{vec}(BA^T),\,v_{N,k}\rangle\,\end{pmatrix}^T \; \in \; \C^{N}.
\end{align*}
Moreover, for every $k\in \{1,\dots,\alpha\}$ we define 
\begin{align*}
{\bf r}(\star_k)={\bf r}(k)&:=\max\big\{\rk X: X\in \mathcal{U}(k)\big\},\\
\mbox{where}\qquad \mathcal{U}(\star_k)=\mathcal{U}(k)&:=\mathrm{span}\,\big\{\mathrm{vec}^{-1}(v_{i,k}):i=1,\dots,N\big\} \; \subset \; \C^{n\times m}.
\end{align*} 
Finally, for a fixed $k\in \{1,\dots,\alpha\}$ the scalar $1/{\min\big(\rk AA^*,\rk BB^*,{\bf r}(k)\big)}$ is the best possible for $AA^* \star_k BB^*$, i.e., it can not be improved over all $A,B$.
\end{utheorem}

As an immediate application, Theorem \ref{T:Schur-ext} yields Theorems \ref{T:schur+jury} and \ref{khare2021sharp}, as well as Theorems \ref{T:kronecker-ext} and \ref{T:jury-ext}. 

\begin{cor}\label{cor:special-cases}
\
\begin{enumerate}
    \item Theorem~\ref{T:Schur-ext} yields Theorem~\ref{T:schur+jury} and Theorem~\ref{khare2021sharp}.
    \item Theorem~\ref{T:Schur-ext} yields the novel stronger forms for $\star \in \{\jury,\otimes\}$ mentioned in Theorem~\ref{T:kronecker-ext} and Theorem~\ref{T:jury-ext}.
\end{enumerate}
\end{cor}

\section{Proofs over Euclidean spaces}\label{S:Euclidean}

The aim of this section is to present proofs of the results in the Euclidean setting, thereby also building intuition for the more intricate Hilbert space related developments in the subsequent sections.

\subsection{General cases of all \(\star \in \prd\)} We begin with the proof of Proposition~\ref{T:prod-convex} and Theorem~\ref{T:Schur-ext}.

\begin{proof}[Proof of Proposition~\ref{T:prod-convex}]

Since positive semidefinite matrices form a closed convex cone the assertion $(a)$ holds. We will prove the other two assertions together. Suppose we are given a product $\star=\star(\mathcal{Y})$ for some positive semidefinite $\mathcal{Y}=(Y_{ij})_{i,j=1}^{N}\in \C^{mnN\times mnN}$, where each $Y_{ij}\in \C^{mn\times mn}$. Since positive semidefinite matrices are Gram matrices, we have that $\mathcal{Y}$ can be written as a product $\mathcal{Y}=\mathcal{X}\mathcal{X}^*$, where $\mathcal{X}\in \C^{mnN\times \alpha}$ for some integer $\alpha\geq 1$. We write $\mathcal{X}$ as a block column matrix to obtain the following: 
\[
\mathcal{Y}
=\mathcal{X}\mathcal{X}^*=
\begin{bmatrix}
X_1\\
X_2\\
\vdots\\
X_N
\end{bmatrix}
\begin{bmatrix}
X_1^* & X_2^* & \cdots & X_N^*
\end{bmatrix}
=
\begin{bmatrix}
X_1 X_1^* & X_1 X_2^* & \cdots & X_1 X_N^*\\
X_2 X_1^* & X_2 X_2^* & \cdots & X_2 X_N^*\\
\vdots & \vdots & \ddots & \vdots\\
X_N X_1^* & X_N X_2^* & \cdots & X_N X_N^*
\end{bmatrix}.
\]
In particular, each block satisfies $Y_{ij} = X_i X_j^*$. Now suppose $v_{i1},\dots,v_{i\alpha}\in \C^{mn}$ are the columns of $X_{i}$ for $i=1,\dots,N$. Using linearity,  we have
\begin{align*}
    (A\star B)_{ij}&= 
    \tr\Big{(}(A\otimes B) \sum_{\kappa=1}^{\alpha} v_{j\kappa} v_{i\kappa}^*\Big{)} = \sum_{\kappa=1}^{\alpha}\tr\big{(}(A\otimes B) v_{j\kappa} v_{i\kappa}^*\big{)}\\
& = (A\star_1 B)_{ij} + (A\star_2 B)_{ij}+ \dots + (A\star_\alpha B)_{ij}
\end{align*}
where $\star_k=\star_k(\mathcal{Y}_k)$ with
\[
\mathcal{Y}_{k}=
\begin{bmatrix}
v_{1\kappa}\\
v_{2\kappa}\\
\vdots\\
v_{N\kappa}
\end{bmatrix}
\begin{bmatrix}
v_{1\kappa}^* & v_{2\kappa}^* & \dots & v_{N\kappa}^*
\end{bmatrix}.
\]
Clearly each of them is rank-one and positive semidefinite. These steps are real-linearly reversible, completing the proof.
\end{proof}

\begin{defn}[The square root]
Let $P\in \C^{\kappa\times \kappa}$ be a positive semidefinite matrix with spectral decomposition $P=U^*DU$ with $U$ unitary and $D=\diag(d_1,\dots,d_\kappa)$ diagonal with nonnegative diagonal entries. We define the {\it square root} of $P$ by $P^{1/2}:=U^*D^{1/2}U$ with $D^{1/2}:=\diag(d_1^{1/2},\dots,d_\kappa^{1/2})$.    
\end{defn}

We can now prove Theorem \ref{T:Schur-ext}.
\begin{proof}[Proof of Theorem~\ref{T:Schur-ext}] We divide the proof into three part.\medskip

\textbf{Part I. Positivity.} The positive semidefinite matrix $\mathcal{Y}$ is a Gram matrix, i.e., it can be written as a product $\mathcal{Y}=\mathcal{X}\mathcal{X}^*$, where $\mathcal{X}\in \C^{mnN\times \alpha}$. Write $\mathcal{X}$ as a block column matrix $\mathcal{X}=\begin{pmatrix}X_1^T & \dots & X_{N}^T\end{pmatrix}^T$ where each $X_i \in \C^{mn\times \alpha}$; which gives that $Y_{ij}=X_iX_j^*$. Let $P\in \C^{m\times m},Q\in \C^{n\times n}$ be positive semidefinite. Then $P\otimes Q$ is positive semidefinite, and so using its square root,
\begin{align*}
P\star Q 
&= \begin{pmatrix} \langle P\otimes Q, X_iX_j^* \rangle  \end{pmatrix}_{i,j=1}^{N}= \begin{pmatrix} \tr ((P\otimes Q) X_jX_i^*)  \end{pmatrix}_{i,j=1}^{N}\\
&= \begin{pmatrix} \tr ((P\otimes Q)^{1/2} X_j ((P\otimes Q)^{1/2} X_i)^*)  \end{pmatrix}_{i,j=1}^{N}\\
&=
\begin{pmatrix} \langle (P\otimes Q)^{1/2} X_j, (P\otimes Q)^{1/2} X_i \rangle  \end{pmatrix}_{i,j=1}^{N}.
\end{align*}
This is a Gram matrix, proving part $(a)$.\medskip

\textbf{Part II. Inequality.} For the proof of part~\textup{(b)}, assume first that $\mathcal{Y}=\mathcal{Y}(\star)$ is
rank-one positive semidefinite. Then there exist vectors $v_1,\dots,v_N\in\C^{mn}$ such
that $Y_{ij}=v_i v_j^*$. Suppose $C_j \in \C^{n\times m}$ is the unique matrix such that
\[
\mathrm{vec}(C_j) = v_j,
\]
where $\mathrm{vec}:\C^{n\times m}\to\C^{mn}$ denotes the column-stacking isometry. Let $A\in \C^{m\times \beta}$ and $B\in \C^{n\times \beta}$ and compute the entries of $AA^*\star BB^*$:
\[
(AA^*\star BB^*)_{ij}=\left\langle (AA^*\otimes BB^*)^{1/2}\mathrm{vec}(C_j),\,
(AA^*\otimes BB^*)^{1/2}\mathrm{vec}(C_i)\right\rangle.
\]
Since $AA^*\otimes BB^*$ is positive and 
\[
AA^*\otimes BB^* = (A\otimes B)(A^*\otimes B^*) = (A\otimes B)(A\otimes B)^*,
\]
we have
\begin{align*}
\langle (AA^*\otimes BB^*)^{1/2}&\mathrm{vec}(C_j),\,
(AA^*\otimes BB^*)^{1/2}\mathrm{vec}(C_i)\rangle \\
&=\left\langle (AA^*\otimes BB^*)\mathrm{vec}(C_j),\,
\mathrm{vec}(C_i)\right\rangle\\
&= \left\langle (A^*\otimes B^*)\mathrm{vec}(C_j),\,
(A^*\otimes B^*)\mathrm{vec}(C_i)\right\rangle.
\end{align*}
The matrix dimensions are compatible for the vectorization identity \cite{vectorization_math_wikipedia}:
\[
(A^*\otimes B^*)\,\mathrm{vec}(C)
= \mathrm{vec}\big(B^* C\,\overline{A}\big),
\qquad C\in\C^{n\times m}.
\]
We thus obtain the Gramian:
\begin{align*}
(AA^*\star BB^*)_{ij}
&=\Big\langle\mathrm{vec}\big(B^*C_j\,\overline{A}\big),\,\mathrm{vec}\big(B^*C_i\,\overline{A}\big)\Big\rangle= \big\langle B^* C_j\,\overline{A},\ B^* C_i\,\overline{A}\big\rangle.
\end{align*}
Thus for all $u=(u_1,\dots,u_N)^T\in \C^{N}$, we have
\begin{align*}
u^*(AA^*\star BB^*)u 
&=\sum_{i,j=1}^{N}\overline{u_i}u_j (AA^*\star BB^*)_{ij} = \sum_{i,j=1}^{N}\overline{u_i}u_j \langle B^* C_j \overline{A}, B^* C_i \overline{A} \rangle\\
& = \bigg\langle B^* \bigg{(}\sum_{j=1}^{N} {u_j}C_j \bigg) \overline{A}, B^* \bigg{(}\sum_{i=1}^{N}{u_i}C_i \bigg) \overline{A} \bigg\rangle = \langle T, T\rangle
\end{align*}
where $T:=B^* \big{(}\sum_{j=1}^{N}{u_j}C_j\big{)} \overline{A}$. Now, for $U:=\sum_{j=1}^{N}{u_j}C_j\in \mathcal{U}(\star)$, we have
\begin{align*}
\rk(B^*U\overline{A})&\leq
\min\big(\rk(B^*),\rk(U),\rk(A)\big)\\
&\leq \min\big(\rk(AA^*),\rk(BB^*),{\bf r}(\star)\big)=:r,    
\end{align*}
where ${\bf r}(\star):=\max\big\{\rk X: X\in \mathcal{U}(\star)\big\}$. Notice that the trace
\[
\tr(T) = \sum_{j=1}^{N} u_j \tr\big(B^* C_j\overline{A}\big) = \rho^* u
\]
where
\begin{align*}
\rho
&=\begin{pmatrix} \overline{\tr\big(B^*C_1\overline{A}\big)} & \cdots &\overline{\tr\big(B^*C_N\overline{A}\big)}\end{pmatrix}^T
=\begin{pmatrix} {\tr\big(B^T\overline{C_1}{A}\big)} & \cdots &{\tr\big(B^T\overline{C_N}{A}\big)}\end{pmatrix}^T\\
&=\begin{pmatrix} \big\langle AB^T,C_1^T\big\rangle & \cdots &\big\langle AB^T,C_N^T\big\rangle\end{pmatrix}^T
=\begin{pmatrix} \big\langle BA^T,C_1\big\rangle & \cdots &\big\langle BA^T,C_N\big\rangle\end{pmatrix}^T\\
&=\begin{pmatrix} \big\langle \mathrm{vec}(BA^T),v_1\big\rangle & \cdots &\big\langle \mathrm{vec}(BA^T),v_N\big\rangle\end{pmatrix}^T.
\end{align*}
We give a quick proof of the trace-rank inequality $|\tr(T)|^2\leq {\rk T} \, \langle T,T \rangle$. Let $P$ denote the orthogonal projection onto $\mathrm{im}(T)$ (equivalently,
onto $(\ker T^*)^\perp$). Then $PT=T$, and so using Cauchy--Schwarz we obtain
\[
|\tr(T)|=|\tr(PT)|=|\langle P,\,T^*\rangle| \le \sqrt{\langle P,P \rangle}\; \sqrt{\langle T^*,T^* \rangle}.
\]
Since $P$ is an orthogonal projection of rank $r$, $\langle P,P \rangle=\|P\|^2=r$, as desired. 

Thus, we obtain
\[
\begin{aligned}
u^*(AA^* \star BB^*)u &= \langle T,T\rangle \geq \frac{1}{r} \left|  \rho^* u \right|^2= \frac{1}{r}\,u^* \rho \rho^* u.
\end{aligned}
\]
This shows the inequality for rank one $\mathcal{Y}=\mathcal{Y}(\star)$. The proof for the higher rank cases follow from this and Proposition~\ref{T:prod-convex}.\medskip

\textbf{Part III. Optimality.} Fix $k\in\{1,\dots,\alpha\}$ and consider the rank-one component $\star_k$. Recall that
\begin{align*}
\mathcal{U}(k)
&:=\mathrm{span}\,\big\{\mathrm{vec}^{-1}(v_{i,k}):\,i=1,\dots,N\big\}
\subset \C^{n\times m},\\
\mbox{and}\qquad{\bf r}(k)&:=\max\big\{\rk X:\,X\in \mathcal{U}(k)\big\}.    
\end{align*}
Thus, there exists
\[
U_0\in\mathcal{U}(k)\setminus\{{\bf 0}_{n \times m}\}
\quad\text{with}\quad
\rk(U_0)={\bf r}(k).
\]
We show that the coefficient ${1}/{\min\bigl(\rk AA^*,\,\rk BB^*,\,{\bf r}(k)\bigr)}$ is optimal for the map $(AA^*,BB^*)\mapsto AA^*\star_k BB^*$ with $A\in\C^{m\times\beta}$ and $B\in\C^{n\times\beta}$, for the given (fixed) $\beta\ge 1$. Set
\[
r_0 \;:=\; \min\bigl(\beta,\,m,\,n,\,{\bf r}(k)\bigr).
\]
We will construct matrices $A\in\C^{m\times\beta}$, $B\in\C^{n\times\beta}$ such that $T:=B^* U_0 \overline{A}$ has exactly $r_0$ diagonal entries $1$, and all other entries $0$, yielding equality in the trace-rank inequality. This will imply that no better constant can hold uniformly for all $A,B$.\medskip

\textit{Step 1: SVD of $U_0$.} Take the singular value decomposition of $U_0$:
\[
U_0 = S\,\Sigma\,R^*,
\]
where $S\in\C^{n\times{\bf r}(k)}$ and $R\in\C^{m\times{\bf r}(k)}$ have orthonormal columns, and $\Sigma = \diag(\sigma_1,\dots,\sigma_{{\bf r}(k)})$ with $\sigma_i>0$. Write
\[
S = \bigl[S_{r_0}\ \ S'\bigr],\qquad
R = \bigl[R_{r_0}\ \ R'\bigr],
\]
where $S_{r_0}\in\C^{n\times r_0}$ and $R_{r_0}\in\C^{m\times r_0}$ collect the first $r_0$
left and the remaining (on the right) singular vectors, and similarly decompose
\[
\Sigma = \begin{pmatrix}
\Sigma_{r_0} & 0 \\
0 & \Sigma'
\end{pmatrix},
\qquad
\Sigma_{r_0}=\diag(\sigma_1,\dots,\sigma_{r_0})\in (0,\infty)^{r_0\times r_0}.
\]

\textit{Step 2: Choice of $A,B$, and the corresponding $T$.} Define matrices $\widetilde{A}\in\C^{m\times r_0}$ and $\widetilde{B}\in\C^{n\times r_0}$ by
\begin{align*}
\widetilde{B} := S_{r_0}\,\Sigma_{r_0}^{-1/2}\quad \mbox{and}\quad
\widetilde{A} := \overline{\,R_{r_0}\,\Sigma_{r_0}^{-1/2}}, 
\end{align*}
so that 
\[
\widetilde{B}^* = \Sigma_{r_0}^{-1/2} S_{r_0}^* \quad \mbox{and}\quad
\overline{\widetilde{A}} = R_{r_0}\,\Sigma_{r_0}^{-1/2}.    
\]
Using $S_{r_0}^*S=[I_{r_0}\ 0]$ and $R^*R_{r_0}=\begin{bmatrix}I_{r_0}\\0\end{bmatrix}$, we get,
\begin{align*}
\widetilde{B}^* U_0\,\overline{\widetilde{A}}=\Sigma_{r_0}^{-1/2}\,S_{r_0}^*\,U_0\,R_{r_0}\,\Sigma_{r_0}^{-1/2}&=\Sigma_{r_0}^{-1/2}\,S_{r_0}^*\,S\,\Sigma\,R^*R_{r_0}\,\Sigma_{r_0}^{-1/2} \\
&= \Sigma_{r_0}^{-1/2}\,\Sigma_{r_0}\,\Sigma_{r_0}^{-1/2}
= I_{r_0}.    
\end{align*}
Now embed $\widetilde{A},\widetilde{B}$ into $A,B$ having exactly $\beta$ columns by padding
with zero columns:
\[
A := \bigl[\ \widetilde{A}\ \ 0_{m\times(\beta-r_0)}\ \bigr] \in \C^{m\times\beta},
\qquad
B := \bigl[\ \widetilde{B}\ \ 0_{n\times(\beta-r_0)}\ \bigr] \in \C^{n\times\beta}.
\]
Then $\rk A = \rk \widetilde{A} = r_0$ and $\rk B = \rk \widetilde{B} = r_0$, and so
\[
\min\bigl(\rk AA^*,\,\rk BB^*,\,{\bf r}(k)\bigr)
= \min(r_0,r_0,{\bf r}(k)) = r_0.
\]
By construction of $A,B$, we have
\[
T
= B^* U_0\,\overline{A}
= \begin{bmatrix}
\widetilde{B}^* \\ 0
\end{bmatrix}
U_0
\Bigl[\ \overline{\widetilde{A}}\ \ 0\Bigr]
=
\begin{bmatrix}
\widetilde{B}^* U_0\,\overline{\widetilde{A}} & 0\\[2pt]
0 & 0
\end{bmatrix}
=
\begin{bmatrix}
I_{r_0} & 0\\[2pt]
0 & 0
\end{bmatrix}
\in\C^{\beta\times\beta}.
\]
Thus $\rk(T) = r_0,$ $\langle T,T\rangle = \tr(T^*T) = \tr(T) = r_0,$ and $|\tr(T)|^2 = r_0^2$. Hence the trace-rank inequality holds with equality $|\tr(T)|^2 = r_0\,\langle T,T\rangle$.\medskip

\textit{Step 3: Sharpness of the constant.} From the proof in Part II, for every $A,B,u$ we have
\[
u^*(AA^*\star_k BB^*)u
= \langle T,T\rangle
\;\ge\;
\frac{1}{r}\,|\tr(T)|^2
= \frac{1}{r}\,u^*\rho_k\rho_k^*u,
\]
where $r := \min\bigl(\rk AA^*,\,\rk BB^*,\,{\bf r}(k)\bigr)$. For the specific choice of $A,B,u$ above, we have $r=r_0$ and equality:
\[
u^*(AA^*\star_k BB^*)u
= \langle T,T\rangle
= \frac{1}{r_0}\,|\tr(T)|^2
= \frac{1}{r_0}\,u^*\rho_k\rho_k^*u.
\]
Thus the constant $1/r_0 = 1/\min\bigl(\rk AA^*,\,\rk BB^*,\,{\bf r}(k)\bigr)$ is attained. More precisely, if one were to replace $1/\min(\rk AA^*,\rk BB^*,{\bf r}(k))$ by any larger constant $c>0$ that is supposed to work for \emph{all} $A,B$, then applied to this particular choice we would obtain
\[
\langle T,T\rangle
= u^*(AA^*\star_k BB^*)u
\;\ge\;
c\,u^*\rho_k\rho_k^*u
= c\,|\tr(T)|^2
= c\,r_0^2.
\]
Since $\langle T,T\rangle = r_0$, this implies $r_0\ge c\,r_0^2$, i.e.
\(
c\le 1/r_0.
\)
Hence no $c>1/r_0$ can work uniformly. Therefore the coefficient ${1}/{\min\bigl(\rk AA^*,\,\rk BB^*,\,{\bf r}(k)\bigr)}$ in Theorem~\ref{T:Schur-ext}(b) is the best possible for the rank-one component $\star_k$.
\end{proof}

\subsection{Specific classical cases of \(\star\in \{\circ,\jury,\otimes\}\)} 

We now prove Theorem~\ref{T:schur,jury,kronecker} and Corollary~\ref{cor:special-cases}.

\begin{proof}[Proof of Theorem~\ref{T:schur,jury,kronecker}]
\textit{Hadamard product $\circ$.} Take $m=n=N$. Suppose $e_1,\dots,e_N$ denote the standard basis of $\C^{N}$. Now, take \(
Y_{ij}:=(e_i\otimes e_{i}) (e_{j}\otimes e_{j})^T,
\) where $\otimes$ denotes the standard Kronecker product. Clearly $\mathcal{Y}$ is positive semidefinite of rank-one. Compute 
\begin{align*}
(A \star B)_{ij} 
&= \tr((A\otimes B) Y_{ij}^*)
= \tr((A\otimes B) (e_j\otimes e_{j}) (e_{i}\otimes e_{i})^T)\\
&=(e_{i}\otimes e_{i})^T (A\otimes B) (e_j\otimes e_{j}) 
=  (e_i^T A e_j) (e_i^T B e_j) 
= a_{ij}b_{ij},
\end{align*}
for all $A=(a_{ij})$ and $B=(b_{ij})\in \C^{N\times N}$, showing that $\circ \in \prd(N,N;N)$.\medskip

\textit{Kronecker product $\otimes$.} Take $m,n\geq 1$ and $N=mn$. Let $e_1,\dots,e_m$ and $f_1,\dots,f_n$ be the standard basis of $\C^m$ and $\C^n$ respectively. Now order the basis elements $e_{i_1}\otimes f_{i_2}$ lexicographically in $(i_1,i_2)$, and take 
\[
Y_{n(i_1-1)+i_2,n(j_1-1)+j_2}:=(e_{i_1}\otimes f_{i_2})(e_{j_1}\otimes f_{j_2})^T
\] 
for all $1\leq i_1,j_1\leq m$ and $1\leq i_2,j_2\leq n$. Clearly $\mathcal{Y}$ is positive semidefinite of rank-one. For $A=(a_{ij})\in \C^{m\times m}$ and $B=(b_{ij})\in \C^{n\times n}$, compute
\begin{align*}
(A \star B)_{n(i_1-1)+i_2,n(j_1-1)+j_2} 
&= \tr((A\otimes B) Y_{n(i_1-1)+i_2,n(j_1-1)+j_2}^*)\\
= \tr((A\otimes B) (e_{j_1}\otimes f_{j_2}) (e_{i_1}\otimes f_{i_2})^T)
&=(e_{i_1}\otimes f_{i_2})^T (A\otimes B) (e_{j_1}\otimes f_{j_2}) \\
=  (e_{i_1}^T A e_{j_1}) (f_{i_2}^T B f_{j_2}) 
&= a_{i_1,j_1}b_{i_2,j_2}.
\end{align*}
Thus $\otimes \in \prd(m,n;mn)$.\medskip

\textit{Convolution $\jury$.} Take $m=n=N$. Suppose $e_1,\dots,e_N$ is the standard basis of $\C^N$. The following yields the rank-one positive semidefinite $\mathcal{Y}(\jury)$:
\[
Y_{ij}:=\bigg[\sum_{k=1}^{i}e_{k}\otimes e_{i-k+1} \bigg] \bigg[\sum_{k=1}^{j}e_{k}\otimes e_{j-k+1} \bigg]^T.
\]
For $A=(a_{ij}),B=(b_{ij})\in \C^{N\times N}$, compute
\begin{align*}
(A \star B)_{ij} 
&= \tr((A\otimes B) Y_{ij}^*)\\
&= \tr\Big((A\otimes B) \Big[\sum_{k=1}^{j}e_{k}\otimes e_{j-k+1} \Big] \Big[\sum_{k=1}^{i}e_{k}\otimes e_{i-k+1} \Big]^T\Big)\\
&=\Big[\sum_{k=1}^{i}e_{k}\otimes e_{i-k+1} \Big]^T(A\otimes B) \Big[\sum_{k=1}^{j}e_{k}\otimes e_{j-k+1} \Big].
\end{align*}
Upon distributing the transpose and cross-multiplying, the above equals:
\begin{align*}
\sum_{k=1}^{i}\sum_{l=1}^{j} (e_k^T\otimes e_{i-k+1}^T) (A\otimes B) (e_{l}\otimes e_{j-l+1})
&=\sum_{k=1}^{i}\sum_{l=1}^{j} (e_k^T Ae_l) (e_{i-k+1}^T B e_{j-l+1})\\
= \sum_{k=1}^{i}\sum_{l=1}^{j} a_{k,l} b_{i-k+1,j-l+1} 
&= (A\jury B)_{i,j}.
\end{align*}
Thus $\jury \in \prd(N,N;N)$.
\end{proof}

We are now ready to prove Corollary~\ref{cor:special-cases}.

\begin{proof}[Proof of Corollary~\ref{cor:special-cases}]
Theorem~\ref{T:schur,jury,kronecker} shows that Theorem~\ref{T:schur+jury} is a special case. To address the remaining cases, using the previous proof of Theorem~\ref{T:schur,jury,kronecker}, we compute $\rho_{\star}=\begin{pmatrix}\rho_1,\dots,\rho_N\end{pmatrix}^T\in \C^{N}$ for $\star\in \{\circ,\jury,\otimes\}$.\medskip

\noindent\textit{Hadamard product $\circ$.} We have $m=n=N$, and $Y_{ij} = (e_i\otimes e_i)(e_j\otimes e_j)^T$. Thus, $v_i = e_i\otimes e_i$. For the column-stacking isometry $\mathrm{vec}:\C^{N\times N}\to \C^{N^2}$, for $A,B\in \C^{N\times \beta}$,
\begin{align*}
\rho_i 
&= \langle \mathrm{vec}(BA^T), v_i \rangle = \langle \mathrm{vec}(BA^T), e_i \otimes e_i \rangle = (BA^T)_{ii}=(AB^T)_{ii}. 
\end{align*}
This yields Theorem~\ref{khare2021sharp}.\medskip

\noindent\textit{Kronecker product $\otimes$.} We have $m,n\geq 1$ and $N=mn$. Recall that $Y_{n(i_1-1)+i_2,n(j_1-1)+j_2}:=(e_{i_1}\otimes f_{i_2})(e_{j_1}\otimes f_{j_2})^T$. Thus, $v_{n(i_1-1)+i_2}
= e_{i_1}\otimes f_{i_2}$ for all $1\leq i_1,j_1\leq m$ and $1\leq i_2,j_2\leq n$, where $(i_1,i_2)$ are ordered lexicographically. Consequently, using the column-stacking vectorization $\mathrm{vec}:\C^{n\times m} \to \C^{mn}$, for $A\in \C^{m\times \beta}$ and $B\in \C^{n\times \beta}$, we have
\begin{align*}
\rho_{\,n(i_1-1)+i_2} 
= \langle \mathrm{vec}(BA^T), e_{i_1} \otimes f_{i_2} \rangle
= (BA^T)_{i_2,i_1}.
\end{align*}
Therefore, since $(i_1,i_2)$ is ordered lexicographically, we have
\[
\rho = \mathrm{vec}(BA^T).
\]

\noindent\textit{Convolution $\jury$.} We have $m=n=N$, and from the proof of Theorem~\ref{T:schur,jury,kronecker}, 
\[
Y_{ij}
:= \bigg[\sum_{k=1}^{i} e_k\otimes e_{i-k+1}\bigg]
   \bigg[\sum_{l=1}^{j} e_l\otimes e_{j-l+1}\bigg]^T.
\]
Thus, $v_i = \sum_{k=1}^{i} e_k\otimes e_{i-k+1}$. Via the column-stacking $\mathrm{vec}:\C^{n\times m}\to \C^{mn}$, we have the desired
\begin{align*}
\rho_i
= \langle \mathrm{vec}(BA^T), v_i \rangle 
&= \sum_{k=1}^{i}\langle \mathrm{vec}(BA^T), e_k \otimes e_{i-k+1} \rangle = \sum_{k=1}^{i} (BA^T)_{i-k+1,k}.
\end{align*}
It is easy to see that ${\bf r}(\star)\geq \rk AA^*, \rk BB^*$ in all cases of $\star\in \{\circ,\otimes,\jury\}$, and so ${\bf r}(\star)$ does not appear in the final inequality. These computations prove Theorems~\ref{T:kronecker-ext} and \ref{T:jury-ext}.
\end{proof}

\section{Qualitative analysis over Hilbert spaces}\label{S:Hilbert}

\subsection{Preliminaries}\label{S:Hilbert-prelim} We begin by recalling several standard concepts.\medskip

Throughout the remainder of this paper, all Hilbert spaces are assumed to be complex. For a Hilbert space $H$, we write $\langle \cdot,\cdot\rangle_{_H}$ for
its inner product, which is linear in the first variable and conjugate-linear in the second. The subscript is omitted if the space is clear from the context. 

Let $H,K$ by Hilbert spaces. We denote by $\Bou_1(H,K) \subset\Bou_2(H,K) \subset \Bou(H,K)$ the trace class, the Hilbert--Schmidt, and the bounded linear operators $:H\to K$, respectively. We write
$\Bou_{\mu}(H):=\Bou_{\mu}(H,H)$ for each $\Bou_\mu\in\{\Bou_1,\Bou_2,\Bou\}$. We use $\lVert\cdot\rVert$, $\lVert\cdot\rVert_2$ (or sometimes
$\lVert\cdot\rVert_{_{\Bou_2(H,K)}}$ for more clarity), and $\lVert\cdot\rVert_1$ to denote the operator norm on $\Bou(H,K)$, the Hilbert--Schmidt norm on $\Bou_2(H,K)$, and the trace norm on $\Bou_1(H,K)$, respectively \cite{conway2000course,kadison1997fundamentals,weidmann1980linear}. For the trace-class operators $:H\to H$, we write $\Tr_{_H}$ for the trace. When the context is clear, the subscripts are omitted. We denote the adjoint of $Q\in \Bou(K,H)$ by $Q^*\in \Bou(H,K)$, so that
\[
\langle Qx,y\rangle_{_{H}} = \langle x,Q^*y\rangle_{_{K}} \qquad \forall x\in K,~y\in H.
\]\

\textit{Rank-one operators.} For $y \in H$ and $x \in K$, we denote by $\ro_{x,y}$ the rank-one operator $\ro_{x,y}: H\to K$ defined by
\[
\ro_{x,y}w := \langle w,y\rangle_{_{H}}\, x \qquad  \forall w\in H.
\]
Thus, the map $(x,y)\mapsto\ro_{x,y}$ is linear in $x$ and conjugate linear in $y$. For further properties, we refer to {\cite[Proposition 16.3]{conway2000course}}.\medskip

\textit{Tensor product of Hilbert spaces and operators.} The algebraic tensor product of $H$ and $K$ is denoted by $H\otimes_{\mathrm{alg}} K$, i.e., the vector space generated by elementary tensors $x\otimes y$, modulo the usual bilinearity
relations, equipped with the sesquilinear form
\[
\langle x_1\otimes y_1,\;x_2\otimes y_2\rangle_{_0}
:= \langle x_1,x_2\rangle_{_{H}}\, \langle y_1,y_2\rangle_{_{K}},
\]
on simple tensors and extended  sesquilinearly \cite[Remark 2.6.7]{kadison1997fundamentals}. The Hilbert tensor product $H\otimes K$ is defined to be its completion.
We write $\langle\cdot,\cdot\rangle_{_{H\otimes K}}$ for the corresponding completed inner product.

The tensor product of $P\in \Bou(H),Q\in \Bou(K)$, on simple tensors is given by
\[
(P\otimes Q)(x\otimes y):=(Px) \otimes (Qy).
\]
This defines a unique bounded linear map from $H \otimes K$ into itself \cite[Proposition 2.6.12]{kadison1997fundamentals}.

\textit{Conjugate Hilbert space.} 
Given a Hilbert space \(\bigl(H,\langle\cdot,\cdot\rangle_{_H}\bigr)\), we denote its conjugate \(\bigl(H^{\#},\langle\cdot,\cdot\rangle_{_{H^{\#}}}\bigr)\) to be the isometric copy of $H$ with the same addition as in $H$, but with scalar multiplication $\ast$ and inner product on \(H^{\#}\) defined by:
\[
\lambda\ast y := \overline{\lambda}\,y \qquad \mbox{and}\qquad
\langle x,y\rangle_{_{H^{\#}}} := \langle y,x\rangle_{_{H}} \qquad \forall\lambda\in\C,~x,y\in H^{\#}.
\]
With these definitions, \(\langle\cdot,\cdot\rangle_{_{H^\#}}\) is an inner product on \(H^\#\) that is linear in the first variable and conjugate-linear in the second {\cite[p. 131]{kadison1997fundamentals}}.\medskip

\textit{Vectorization isometry.} The space $\Bou_2(H^{\#},K)$ is equipped with the Hilbert--Schmidt inner product
\[
\langle X,Y\rangle_{_{\Bou_2(H^{\#},K)}}
:= \Tr_{_{K}}(X Y^{*})= \Tr_{_{H^\#}}(Y^{*}X)
\qquad \forall X,Y\in \Bou_2(H^{\#}, K).
\]
There exists an isometry between $\Bou_2(H^{\#},K)$ and $H\otimes K$:
\[
\mathrm{vec}:\Bou_2(H^{\#},K)\longrightarrow H\otimes K  \qquad \mbox{defined by} \qquad \mathrm{vec}(\ro_{y, \overline{x}}):= x\otimes y
\]
on rank-one operators $\ro_{y,\overline{x}}:H^\# \to K$, extended linearly and continuously \cite[Proposition 2.6.9] {kadison1997fundamentals}.\medskip

\textit{Conjugation isometry.} The conjugation isometry is given by
\[
\conj_H : H \longrightarrow H^{\#} \qquad\mbox{defined by}
\qquad y \longmapsto \overline{y},
\]
where \(\overline{y}\) denotes the same vector \(y\), viewed as an element of \(H^{\#}\) (with its modified scalar multiplication and inner product).  
Since \(H\) and \(H^\#\) have the same underlying additive structure, \(\conj_H\) is the identity map on the additive group and is therefore automatically additive. 
Moreover, for $\lambda\in\C,x\in H$,
\begin{align*}
\conj_H(\lambda x)
   = \overline{\lambda x}
   = \lambda x
   = \overline{\lambda}\ast x
   = \overline{\lambda}\ast \overline{x}
   = \overline{\lambda}\ast \conj_H(x),    
\end{align*}
where \(\overline{\lambda x} = \lambda x\) because \(\conj_H\) is the identity on the underlying set. Thus \(\conj_H(\lambda x) = \overline{\lambda}\ast \conj_H(x)\) and so \(\conj_H\) is conjugate linear. It is clear that \(\conj_H\) is an isometry. One may refer to {\cite[Corollary 2.3.2]{kadison1997fundamentals}} for more details on this operator defined from $H$ into its Banach dual.
\medskip

\textit{Transposition operators.} Given $Q\in \Bou(K,H)$, we define a corresponding transposition operator
\[
Q^\#:H^\# \longrightarrow K^\#  \qquad \mbox{given by}\qquad Q^\#:= \conj_{K}\, Q^*\, \conj_{H}^{-1}. 
\]
Then $Q^{\#}$ is bounded. See Lemma~\ref{lemma:transposition} and {\cite[p. 102]{kadison1997fundamentals}} for further properties of this operation.\medskip

\textit{Rank of bounded linear operators.}
The rank of $P\in \Bou(K,H)$, denoted by $\rk P$, is the dimension of the closed range of $P$. Thus $\rk P = \rk P^* = \rk (PP^*) = \rk P^\#$.\medskip

\textit{Positive operators and Loewner order $\succcurlyeq $.} An operator $P\in \Bou(H)$ is said to be positive if $\langle Px,x\rangle \geq 0$ for all $x\in H$ {\cite[\S 3]{conway2000course}}. These operators are self-adjoint. We say that $P \succcurlyeq Q$ for self-adjoint operators $P,Q\in \Bou(H)$ if $P-Q$ is positive. The resulting ordering is called the Loewner order. 

\subsection{Main results over Hilbert spaces} We are now ready to discuss the construction of positive bilinear products over Hilbert spaces.

\begin{defn}\label{defn:prd-hil}
Suppose $H_1,H_2,K$ are given Hilbert spaces. Define 
\[
H=H_1\otimes H_2 \qquad \mbox{and}\qquad  \Hil=K\otimes H.
\]
Fix $\mu\in \{1,2\}$ and let $T\in \Bou_\mu(\Hil)$.
\begin{enumerate}[$(1)$]
\item For an orthonormal basis $\{e_i\}_{i\in I}$ of $K$, consider ``slices'' $T_{ij}:=T_{e_i,e_j}$ of $T$ defined by:
\[
\langle T_{ij}\eta,\, \xi\rangle_{_{H}}:=\langle T(e_j\otimes \eta),\,e_i\otimes \xi\rangle_{_{\Hil}}
\qquad(\xi,\eta\in H).
\]
Then each $T_{ij}\in \Bou_\mu(H)$. (Proposition~\ref{prop:slicing+stitching})
\item Using the slices, define a bilinear product $\star=\star(T)$:
\[
\star:\Bou_2(H_1)\times \Bou_2(H_2) \longrightarrow \Bou(K)
\] 
where, for $P\in \Bou_2(H_1)$ and $Q\in \Bou_2(H_2)$, we define $P\star Q\in \Bou(K)$ by
\begin{align*}
\langle P\star Q e_j,e_i\rangle_{_{K}} :=\langle P\otimes Q, T_{ij}\rangle_{_{\Bou_2(H)}}. 
\end{align*}
Since both $P\otimes Q$ and $T_{ij}\in \Bou_2(H)$, the above is well-defined.
\end{enumerate}
Now, pick the products given by positivity:
\begin{align*}
&\prd(H_1,H_2;K;\mu;\{e_i\}_{i\in I}):=\\
& \big\{\star=\star(T):\Bou_2(H_1)\times \Bou_2(H_2) \longrightarrow \Bou(K) ~\big| ~T\in \Bou_\mu(H) \mbox{ is positive}\big\}.
\end{align*}
If the basis remains the same in a discussion, then we use $\prd(H_1,H_2;K;\mu)$. Similarly, we write $\prd(\mu)$ if $H_1,H_2,K$ remain the same as well.
\end{defn}

The following is the main result.

\begin{utheorem}\label{T:Schur-ext-Hil}
Following the notation in Definition~\ref{defn:prd-hil}, let $H_1,H_2,K$ be Hilbert spaces, and let $\{e_i\}_{i\in I}$ be an orthonormal basis of $K$. Then the following holds.
\begin{enumerate}[$(a)$]
    \item If $\star=\star(T)\in \prd(\mu)$ for $T\in \Bou_\mu(\Hil)$ positive, then for all $P\in \Bou_2(H_1)$ and $Q\in \Bou_2(H_2)$, the operator $P\star Q \in \Bou_\mu(K)$ and satisfies
\[
\|P\star Q\|_{\mu}\; \leq\;   \|P\|_{2}\,\|T\|_{\mu}\,\|Q\|_{2} \qquad (\mu=1,2).
\]
Furthermore,
\[
P \; \succcurlyeq \; 0 \; \mbox{ and }\; Q \; \succcurlyeq \; 0 \quad \implies \quad P\star Q\;\succcurlyeq\; 0.
\]

\item More strongly, if $T$ is nonzero positive trace-class, then for all nonzero $A\in \Bou_2({\widetilde\Hil},H_1)$ and $B\in \Bou_2({\widetilde\Hil}^{\#},H_2)$, where $\widetilde{\Hil}$ is a given Hilbert space, we have the following trace-class lower bound in the Loewner order $\succcurlyeq$:
\begin{align*}
AA^*\star BB^* \;  \succcurlyeq \; \sum_{n=1}^{\infty} \frac{\lambda_n}{\min\big(\rk AA^*,\,\rk BB^*,\,{\bf r}(n)\big)} \; \ro_{\rho_n,\rho_n} \;\succcurlyeq \; 0,
\end{align*}
where we follow the convention $\frac{1}{\infty}:=0$.
\end{enumerate}
The sequence of triples $\big((\lambda_n,{\bf r}(n),\rho_{n})\big)_{n\geq 1}$ depends on $A,B$ and $T$ via the following spectral resolution into positive rank-one orthogonal operators:
\begin{align*}
T=\sum_{n=1}^{\infty} \lambda_n \ro_{w_n,w_n} \quad \mbox{where }  w_n = \sum_{i\in I} e_i \otimes u_i^{(n)} \quad \mbox{with }  u_i^{(n)}\in H_1\otimes H_2.
\end{align*}
Using the vectorization and transposition, we define the following for $n\geq 1$:
\begin{align*}
    \rho_n &\; := \; \sum_{i\in I} \; \langle \, \mathrm{vec}(BA^\#),u_i^{(n)}\,\rangle_{_{H_1\otimes H_2}}\; e_i \; \in \; K,\\
    \mbox{and}\qquad {\bf r} (n)&\; :=\; \sup \big\{\rk X: X \,\in\, \mathcal{U}(n)\big\},\\
    \mbox{where}\qquad \mathcal{U}(n)&\; :=\; \overline{\mathrm{span}}^{\Bou_2}\,\big\{\mathrm{vec}^{-1}(u_i^{(n)}):i\,\in\, I\big\} \; \subset \; \Bou_2(H_1^{\#},H_2).
\end{align*}
\end{utheorem}

\begin{remark}
Using Theorem~\ref{T:Schur-ext-Hil}, we will later prove Theorem~\ref{T:Hilbert-tensor}.
\end{remark}

\begin{remark}[Canonical formulation]
While the aforementioned basis-dependent approach is essential for the analysis, we emphasize that the resulting constructions are ultimately canonical: in Section~\ref{S:canonical-form} we present this for all rank-one admissible operators; see Corollary~\ref{cor:canonical}.
\end{remark}

\begin{remark}
Several features of Theorem~\ref{T:Schur-ext-Hil}, which appear automatic in Theorem~\ref{T:Schur-ext}, require explicit reformulation in the operator-theoretic setting. We briefly highlight two such points.\medskip

\textit{Vectorization.}
In Theorem~\ref{T:Schur-ext}, the column-stacking vectorization from $\C^{n\times m}\to\C^{mn}$, which is often treated as a coordinate-level convention, tacitly identifies $\C^{m}$ with its dual and suppresses the role of conjugation. In the Hilbert space setting, however, this identification can no longer be made implicitly. The correct operator theoretic analogue of $\C^{n\times m}$ is the Hilbert space $\Bou_2(H_1^{\#},H_2)$, and the canonical isometry
$\mathrm{vec}:\Bou_2(H_1^{\#},H_2)\to H_1\otimes H_2$ provides the corresponding vectorization.
\medskip

\textit{Transposition.} In the finite-dimensional Theorem~\ref{T:Schur-ext}, the matrix transpose $A^{T}$ appears in vectorization through $\mathrm{vec}(BA^{T})$. As discussed above, once vectorization is interpreted as $\mathrm{vec}:\Bou_2(H_1^{\#},H_2)\to H_1\otimes H_2$, the definition of $\mathrm{vec}$ itself necessitates operators acting on the conjugate space $H_1^{\#}$. Thus, in Theorem~\ref{T:Schur-ext-Hil} the finite-dimensional $A^{T}$ is replaced by the canonical $A^{\#}$.
\end{remark}

\section{Proofs over Hilbert spaces}\label{S:Hilbert-proofs}

\subsection{Slicing and stitching}
In this subsection, we formalize the procedure of extracting submatrices in the Hilbert space setting, a process we refer to as \emph{slicing} an operator. Having identified these slices, we then describe a method for reconstructing the original operator from them, which we call \emph{stitching}. While the stitching part is not needed later in the paper, we include it for completeness.

\begin{prop}\label{prop:slicing+stitching}
Suppose $H_1,H_2,K$ are given Hilbert spaces, and let $H:=H_1\otimes H_2$ and $\Hil:=K\otimes H$. Fix an operator class $\Bou_\mu\in\{\Bou_1, \Bou_2, \Bou\}$. Let $T\in \Bou_\mu(\Hil)$, and for $u,v\in K$, define $T_{u,v}:H\to H$ via 
\[
\langle T_{u,v}\eta,\, \xi\rangle_{_H}
:= \langle T(v\otimes \eta),\,u\otimes \xi\rangle_{_{\Hil}}
\qquad (\xi,\eta\in H).
\]
Then the following holds.
\begin{enumerate}[1.]
\item \emph{(Slicing.)} For any $u, v \in K$, we have $T_{u,v} \in \Bou_\mu(H)$ and 
\[
\|T_{u,v}\|_{\mu}\le \|T\|_{\mu}\,\|u\|\,\|v\|.
\] 
The map $(u,v)\mapsto T_{u,v}$ is conjugate-linear in $u$, linear in $v$,
and jointly bounded $:K\times K\to \Bou_\mu(H)$. Moreover, if $\big(T^{(n)}\big)_{n\geq 1}\subset \Bou_{\mu}(\Hil)$ and $T=\sum_{n=1}^{\infty}T^{(n)}$ converges in $\Bou_{\mu}(\Hil)$, then for all $u,v\in K$,
\[
T_{u,v}=\sum_{n=1}^{\infty} T^{(n)}_{u,v} \quad \mbox{convergence in $\Bou_{\mu}(H)$}.
\]
\item \emph{(Stitching.)} For an orthonormal basis $\{e_i\}_{i\in I}$ of $K$ suppose $T_{ij}:=T_{e_i,e_j}$. Then $T\in \Bou_\mu(\Hil)$ can be re-constructed from the $T_{ij}$s as follows. Define $Q$ on finite sums:
\[
Q \Big[\sum_{j\in F} e_j\otimes y_j\Big]
:= \sum_{i\in I} e_i\otimes \Big[\sum_{j\in F} T_{ij}(y_j)\Big],
\qquad F \subset I \text{ finite},~ y_j\in H.
\]
Then $Q$ extends to a bounded linear operator on $\Hil$, and the extended $Q=T$. In particular $Q=T\in\mathcal{B}_{\mu}(\Hil)$ and $\|Q\|_{\mu}=\|T\|_{\mu}$ for the chosen operator class $\Bou_{\mu}\in \{\Bou,\Bou_{1},\Bou_2\}$.
\end{enumerate}
\end{prop}

\begin{proof} \textit{Part 1. Slicing.} For $u\in K$, let 
\[
J_{u}:H\to K \otimes H \quad \mbox{defined by} \quad J_{u}\xi:=u\otimes \xi \qquad (\xi \in H).
\]
Then $J_u$ is linear with the operator norm $\lVert J_u \rVert = \lVert u \rVert$. By definition,
\[
\langle T_{u,v} \eta, \xi \rangle_{_{H}} = \langle T(v\otimes \eta),u \otimes \xi \rangle_{_{\Hil}} = \langle T J_{v}\eta,J_{u}\xi\rangle_{_{\Hil}} = \langle J_{u}^* T J_{v}\eta, \xi\rangle_{_{H}}.
\]
Hence, $T_{u,v}= J_{u}^* T J_{v}$. Therefore, we have $ \lVert T_{u,v} \rVert \leq \lVert J_{u}^* \rVert~ \lVert  T \rVert~ \lVert J_{v} \rVert$. In the cases where $\mu \in \{1,2\}$, the operator classes $\Bou_\mu$ are two-sided ideals with the property $\lVert AXB \rVert_{\mu} \leq \lVert A \rVert \lVert X \rVert_{\mu} \lVert B \rVert$, where $A,B$ are bounded; see e.g.~ {\cite[p. 141]{kadison1997fundamentals}}.
This shows the norm bounds. Linearity and conjugate-linearity are clear. The final part of the first assertion follows from the fact that $T_{u,v} = J_u^* T J_v$ and continuity.\medskip

\textit{Part 2. Stitching.} Recall that given a Hilbert space $H$ and an index set $I$, the Hilbert space of square summable sequences is
\[
\ell^2(I;H):=\big\{ (x_{i})_{i\in I}:x_i\in H \mbox{ and } \sum_{i\in I}\lVert x_i \rVert^2 <\infty \big\} \; \cong \; K\otimes H.
\]
The adjoint of $J_u$ defined above is given by $J_u^*(w\otimes \eta)=\langle w,u\rangle \eta$ on simple tensors. Therefore,
\[
J_{e_i}^*(e_k\otimes \eta)=\delta_{ik} \eta \qquad (i\in I,~\eta\in H).
\]
If $y=\sum_{k\in I} e_k\otimes \eta_k$ is finitely supported, then $J_{e_k}^*y=\eta_k$ for $k\in I$ and $0$ otherwise. Hence
\begin{equation}\label{eq:Parseval}
\sum_{i\in I}\|J_{e_i}^* y\|_{_H}^2=\sum_{i\in I}\|\eta_i\|_{_H}^2=\|y\|_{_{K\otimes H}}^2.
\end{equation}
By the density of such finite sums and continuity of the norm, the above holds for all $y\in K\otimes H$.\medskip

Suppose $y=\sum_{j\in F} e_j\otimes y_j$ with $F$ finite. For any $k\in I$, using the earlier slice identity $T_{kj}=J_{e_k}^* T J_{e_j}$ we compute
\begin{align*}
J_{e_k}^*(Qy)
&=J_{e_k}^*\sum_{i\in I} e_i\otimes \Big[\sum_{j\in F} T_{ij}(y_j)\Big]\\
&=\sum_{j\in F} T_{kj}(y_j)
=\sum_{j\in F} J_{e_k}^*\,T\,(e_j\otimes y_j)
=J_{e_k}^*(Ty).
\end{align*}
Thus $J_{e_k}^*(Ty-Qy)=0$ for all $k\in I$. Applying \eqref{eq:Parseval} gives
\[
\|Ty-Qy\|^2_{K\otimes H}=\sum_{k\in I}\|J_{e_k}^*(Ty-Qy)\|^2_{H}=0.
\]
Hence $Qy=Ty$ for every finitely supported $y \in \Hil$. Therefore, since $T$ is bounded,
\[
\|Qy\|=\|Ty\|\le \|T\|\,\|y\|.
\]
It follows that $Q=T$ on $K\otimes H$, and $\lVert Q\rVert = \lVert T\rVert$.\medskip

To prove the Schatten-class stability, let $\mathcal{J}:K\otimes H\to \ell^2(I;H)$ be the map
\[
\mathcal{J}(z):=\big(J_{e_i}^* z\big)_{i\in I},
\]
which is an isometry by \eqref{eq:Parseval}. Its adjoint is the map
\[
\mathcal{J}^*:\ell^2(I;H)\to K\otimes H,\qquad
\mathcal{J}^*((y_i)_{i\in I})=\sum_{i\in I} e_i\otimes y_i,
\]
and $\|\mathcal{J}\|=\|\mathcal{J}^*\|=1$, with $\mathcal{J}^*\mathcal{J}=\mathrm{Id}_{K\otimes H}$. Define the block operator
\[
S:\ell^2(I;H)\to \ell^2(I;H),\qquad S((y_j)):=\Big(\sum_{j\in I} T_{ij}(y_j)\Big)_{i\in I}.
\]
On the dense subspace of finitely supported vectors, $S=\mathcal{J}\,T\,\mathcal{J}^*$. Hence by density, it holds on all of $\ell^2(I;H)$. If $T\in \Bou_\mu(\Hil)$ with $\mu\in\{1,2\}$, the two-sided ideal property yields (see {\cite[p. 141 for HS]{kadison1997fundamentals}}, {\cite[\S 18]{conway2000course}}, {\cite[Theorem~7.8]{weidmann1980linear}}):
\begin{align*}
S=\mathcal{J}T\mathcal{J}^*\in \Bou_\mu(\ell^2(I;H)) &\implies Q=\mathcal{J}^* S \mathcal{J}\in \Bou_\mu(\Hil)\\
&\implies \|Q\|_\mu \le \|S\|_\mu \le \|T\|_\mu.
\end{align*}
Since $Q=T$, we conclude $\|Q\|_\mu=\|T\|_\mu$.
\end{proof}

\subsection{Transposition and vectorization} We recall some properties of the transposition operator and the vectorization isometry introduced in Subsection~\ref{S:Hilbert-prelim} that will be needed later.

\begin{lemma}\label{lemma:transposition} 
Let $H_1,H_2,K$ be Hilbert spaces. For any bounded operator $Q\in\mathcal{B}(H_1,H_2)$ we have:
\begin{enumerate}[1.]
\item  If $Q\in \Bou_\mu(H_1,H_2)$ for some $\Bou_{\mu}\in \{\Bou,\Bou_{1},\Bou_2\}$, then $Q^\#\in \Bou_\mu(H_2^\#,H_1^\#)$ and $\|Q^\#\|_\mu = \|Q\|_\mu$.
\item If \(R\in\mathcal{B}(K,H_1)\), then $(QR)^\# = R^\#\, Q^\#$.
\item The adjoint and transposition satisfy $(Q^*)^\#=(Q^\#)^* = \conj_{H_2}\, Q\, \conj_{H_1}^{-1}$.
\end{enumerate}
\end{lemma}
\begin{proof} For (1), since both $\conj_{H_1}$ and $\conj_{H_2}$ are conjugate linear and isometries, $Q^\#$ is linear and bounded. The Schatten class stability can be deduced using that both $\conj_{H_1}$ and $\conj_{H_2}$ are (conjugate-linear) isometries.\medskip

To prove $(2)$, let $R\in\mathcal{B}(K,H_1)$. By definition
\begin{align*}
(QR)^\#
= \conj_K\, (QR)^*\, \conj_{H_2}^{-1}
&= \conj_K\, R^* Q^*\, \conj_{H_2}^{-1}\\
&=\bigl(\conj_K R^* \conj_{H_1}^{-1}\bigr)
  \bigl(\conj_{H_1} Q^* \conj_{H_2}^{-1}\bigr)
= R^\# Q^\#.    
\end{align*}

Finally, to prove $(3)$, let $\overline{y}\in H_2^\#$ and $\overline{x}\in H_1^\#$, and compute:
\begin{align*}
\langle Q^\# \overline{y},\ \overline{x}\rangle_{_{H_{1}^\#}} 
&= \langle \conj_{H_1} Q^* \conj_{H_2}^{-1} \overline{y},\ \conj_{H_1} x\rangle_{_{H_1^\#}}= \langle \conj_{H_1} Q^* y,\ \conj_{H_1} x\rangle_{_{H_1^\#}}\\
= \langle x ,Q^* y \rangle_{_{H_1}}
&= \langle Qx , y \rangle_{_{H_2}}
= \langle \conj_{H_2}y,\conj_{H_2}Qx \rangle_{_{H_2^\#}}
= \langle \overline{y}, \conj_{H_2}Q \conj_{H_1}^{-1}\overline{x} \rangle_{_{H_2^\#}}.
\end{align*}
Therefore $(Q^{\#})^*=\conj_{H_2}Q \conj_{H_1}^{-1}$. For $Q^*: H_2 \to H_1$, by definition, we immediately have $(Q^*)^{\#} = \conj_{H_2} Q \conj_{H_1}^{-1}$.
\end{proof}

\begin{lemma}\label{lemma:vec-identity}
Suppose $H_1,H_2,\widetilde{\Hil}$ are Hilbert spaces, and $X\in \Bou_2(H_1^\#,H_2)$.

\begin{enumerate}[1.]
\item For all $u\in H_1$ and $v\in H_2$ the following holds:
\begin{align*}
\langle \mathrm{vec}(X),\, u\otimes v\rangle_{_{H_1\otimes H_2}}
= \big\langle X\,\overline{u},\, v\big\rangle_{_{H_2}}.  
\end{align*}
\item For all $A\in \Bou_2(\widetilde\Hil,H_1)$ and $B\in \Bou_2(\widetilde\Hil^{\#},H_2)$ the following holds:
\begin{align*}
AA^* \otimes BB^* \, \mathrm{vec}(X) = \mathrm{vec}(BB^* X(A^\#)^*A^\#).    
\end{align*}
\end{enumerate}

\end{lemma}
\begin{proof} We begin with the proof of the first assertions. It suffices to prove it for rank-one $X=\ro_{y,\overline{x}}$, where $x\in H_1$ and $y\in H_2$, since finite rank operators are linear combinations of such rank-one operators, and are dense in $\Bou_2(H_1^\#,H_2)$, while both sides of the first assertion define a bounded linear map on $X$. Using $\mathrm{vec}(\ro_{y,\overline{x}})=x\otimes y$, we obtain
\begin{align*}
\langle \mathrm{vec}(\ro_{y,\overline{x}}),\, u\otimes v\rangle_{_{H_1\otimes H_2}} 
&= \langle x,\, u\rangle_{_{H_1}} \langle y,\, v\rangle_{_{H_2}} \\
= \langle \overline{u},\overline{x}\rangle_{_{H_1^\#}} \,\langle y,\, v \rangle_{_{H_2}}
&=\big\langle \ro_{y,\overline{x}}\,\overline{u},\, v\big\rangle_{_{H_2}}.
\end{align*}
This shows the desired identity.\medskip

To show the second identity,  suppose $X=\ro_{y,\overline{x}}$ (rank-one) for $\overline{x}\in H_1^{\#}$ and $y\in H_2$. Then, by definition $\mathrm{vec}(\ro_{y,\overline{x}}) = x\otimes y$, and thus,
\begin{align*}
AA^* \otimes BB^* \mathrm{vec}(\ro_{y,\overline{x}})
&= AA^* \otimes BB^* (x\otimes y) \\
= (AA^*x) \otimes (BB^* y) 
&= \mathrm{vec}(\ro_{BB^*y,\overline{AA^*x}}).    
\end{align*}
On the other hand, for $\overline{z}\in H_1^{\#}$, using Lemma~\ref{lemma:transposition}, we have
\begin{align*}
BB^* \ro_{y,\overline{x}} (A^\#)^* A^\# (\overline{z})
&= BB^* \ro_{y,\overline{x}} \conj_{H_1} A \conj_{\widetilde{\Hil}}^{-1} \conj_{\widetilde{\Hil}} A^* \conj_{H_1}^{-1} (\overline{z})\\
= BB^* \ro_{y,\overline{x}} \overline{AA^* z}
&= \langle \overline{AA^* z}, \overline{x}\rangle_{_{H_1^\#}}  BB^*y
= \langle x, AA^* z\rangle_{_{H_1}}  BB^*y \\
= \langle AA^*x,  z\rangle_{_{H_1}}  BB^*y 
&= \langle \overline{z},\overline{AA^*x}\rangle_{_{H_1^{\#}}}  BB^*y
= \ro_{BB^*y,\overline{AA^*x}} (\overline{z}).
\end{align*}
Thus, using the earlier computation
\[
\mathrm{vec}\big(BB^* \ro_{y,\overline{x}} (A^\#)^* A^\#\big) = \mathrm{vec} \big( \ro_{BB^*y,\overline{AA^*x}} \big) = AA^* \otimes BB^* \mathrm{vec}(\ro_{y,\overline{x}}).
\]
Therefore, the desired identity holds for rank-one $X$, and extends to $\Bou_2(H_1^{\#},H_2)$ by continuity.
\end{proof}

\subsection{Theorem~\ref{T:Schur-ext-Hil} for rank-one \(T=T(\star)\)} In order to prove Theorem \ref{T:Schur-ext-Hil}, we begin by addressing the case of  admissible rank-one positive operators. We start with the following lemma.

\begin{lemma}[Slicing of rank-one $T$]\label{prop:rank-one-slicing}
Let $H_1,H_2,K$ be Hilbert spaces, let $H:=H_1\otimes H_2$, $\Hil:=K\otimes H$, and let an orthonormal basis $\{e_i\}_{i\in I}$ of $K$ be given. For a rank-one positive operator $T=\ro_{w,w}\in \Bou(\Hil)$, write
\[
w=\sum_{i\in I} e_i\otimes u_i,\qquad 
u_i\in H,\qquad \sum_{i\in I}\|u_i\|^2<\infty.
\]
Then for $i,j\in I$, the slices are given by $T_{e_{i},e_j}=\ro_{u_i,u_j}$. For $\star=\star(T)$, we have $P\star Q\in \Bou_1(K)$ for all $P\in \Bou_2(H_1)$ and $Q\in \Bou_2(H_2)$, with
\begin{align*}
\lVert P\star Q \rVert_1~ \leq~ & \lVert P\rVert_2 ~\lVert Q\rVert_2~ \lVert w\rVert^2,\\
\mbox{and}\qquad \langle (P\star Q)e_j,e_i\rangle_{_{K}} ~= ~& \langle (P\otimes Q)u_j,\,u_i\rangle_{_{H}} \qquad (i,j\in I).
\end{align*}
Moreover, $P\star Q\;  \succcurlyeq \;0$ whenever $P\;  \succcurlyeq \; 0$ and $Q\;  \succcurlyeq \;0$.
\end{lemma}

\begin{proof} Since $w=\sum_{k\in I} e_k\otimes u_k$, for any $j\in I$ and $\eta\in H$, by definition
\begin{align*}
T(e_j\otimes \eta)
&=\ro_{w,w} (e_j\otimes \eta)
= \langle e_j\otimes \eta,w\rangle_{_{\Hil}}\, w \\
&= \Big(\sum_{k\in I}\langle e_j, e_k\rangle_{_{K}}\, \langle \eta,u_k\rangle_{_{H}} \Big) \, w 
= \langle \eta,u_j\rangle_{_{H}}\, w.    
\end{align*}
Hence, for $\xi,\eta\in H$, by definition
\begin{align*}
\langle T_{ij}\eta,\xi\rangle_{_{H}}
&= \langle T(e_j\otimes \eta),e_i\otimes \xi\rangle_{_{\Hil}}
= \langle \eta, u_j\rangle_{_{H}} \, \langle w,e_i\otimes \xi\rangle_{_{\Hil}}\\
&= \langle \eta, u_j\rangle_{_{H}} \, \langle u_i,\xi\rangle_{_{H}}
= \big \langle \langle \eta, u_j\rangle_{_{H}}\,u_i,\xi\big\rangle_{_{H}}
= \langle \ro_{u_i, u_j}\eta,\xi\rangle_{_{H}}.
\end{align*}
Therefore $T_{ij}=\ro_{u_i, u_j}$, proving the first claim.\medskip

\noindent Next, from {\cite[ Proposition 16.3]{conway2000course}} we have 
\begin{align*}
\langle(P\star Q)e_j,e_i\rangle_{_{K}}
&:=\Tr_{_{H}} (T_{ij}^* (P\otimes Q)) 
= \Tr_{_{H}}(\ro_{u_j, u_i} (P\otimes Q)) \\
&=\Tr_{_{H}} (\ro_{u_j, (P\otimes Q)^* u_i})
= \langle (P\otimes Q)u_j, u_i\rangle_{_{H}}.    
\end{align*}
Finally, define $V:K\to H$ by $Ve_i := u_i$, and extend it linearly. Then $V\in \Bou_2(K,H)$ since $\lVert V \rVert_2^2 = \sum_{i\in I} \lVert u_i \rVert^2 = \lVert w \rVert^2 <\infty$. Moreover for Hilbert--Schmidt operators $P,Q$, the product $P\otimes Q$ is Hilbert--Schmidt. Since the product of two Hilbert--Schmidt operators is trace-class {\cite[p. 141]{kadison1997fundamentals}}, 
\[
P\star Q = V^*(P\otimes Q)V \in \Bou_1(K) \quad \mbox{and}\quad \lVert P\star Q \rVert_1 \leq \lVert V\rVert_2^2  \lVert P\rVert_2 \lVert Q\rVert_2.
\]
This also shows the final positivity implication, concluding the proof.
\end{proof}

We now state and prove Theorem~\ref{T:Schur-ext-Hil} for rank-one $T=T(\star)$.

\begin{theorem}\label{T:thmC-for-rank1-fixed}
Suppose $H_1,H_2,K$ are Hilbert spaces, and let $H=H_1\otimes H_2$ and $\Hil=K\otimes H$. Suppose $\{e_i\}_{i\in I}$ is an orthonormal basis of $K$. For a rank-one operator $T=\ro_{w,w} \in \Bou(\Hil)$ with $w\in \Hil\setminus\{0\}$, write
\[
w=\sum_{i\in I} e_i\otimes u_i,\qquad 
u_i\in H,\qquad\sum_{i\in I}\|u_i\|^2<\infty.
\]
Let $\widetilde{\Hil}$ be a given Hilbert space, and let $A\in \Bou_2(\widetilde{\Hil},H_1),$ and $B\in \Bou_2(\widetilde{\Hil}^\#,H_2)$ be nonzero. Then
\begin{align*}
    \rho &\; := \; \sum_{i\in I} \; \langle \, \mathrm{vec}(BA^\#),u_i\,\rangle_{_{H_1\otimes H_2}}\; e_i \; \in \; K.
\end{align*}
In fact $\lVert \rho \rVert \leq \lVert A \rVert_2 \,\lVert B \rVert_2\,\sqrt{\lVert T \rVert_2} <\infty$. Moreover, for $\star=\star(T)\in \prd(\{e_i\}_{i\in I})$, we have
\begin{align*}
AA^* \star BB^* \;  \succcurlyeq \; \frac{1}{\min\big(\rk AA^*,\rk BB^*,{\bf r}(\star)\big )} \ro_{\rho,\rho} \;  \succcurlyeq \; 0,
\end{align*}
where we follow the convention the $\frac{1}{\infty}:=0$, and define 
\begin{align*}
    {\bf r}(\star)&\; :=\; \sup \big\{\rk X: X \,\in\, \mathcal{U}(\star)\big\},\\
    \mbox{where}\qquad \mathcal{U}(\star)&\; :=\; \overline{\mathrm{span}}^{\Bou_2}\,\big\{\mathrm{vec}^{-1}(u_i):i\,\in\, I\big\} \; \subset \; \Bou_2(H_1^{\#},H_2).
\end{align*}
Moreover, the scalar $1/\min\big (\rk AA^*,\rk BB^*,{\bf r}(\star)\big )$ is the best possible universal constant, i.e., it cannot be improved uniformly over all $A,B$.
\end{theorem}

\begin{proof} \textbf{Part I. Inequality.} Recall from Lemma~\ref{prop:rank-one-slicing} that
\[
\langle AA^* \star BB^* e_j,e_i \rangle_{_{K}} 
= \langle AA^* \otimes BB^* u_j,u_i \rangle_{_{H}}.
\]\medskip
Let $X_i:=\mathrm{vec}^{-1}(u_i)\in \Bou_2(H_1^{\#},H_2)$.

\textit{Step 1. The Gramian structure.}
Recall that for $u_i=\mathrm{vec}(X_i)$, we have
\begin{align*}
\langle (AA^{*}\star BB^{*})e_{j},e_{i}\rangle_{_{K}}
=
\big\langle
(AA^{*}\otimes BB^{*})\mathrm{vec}(X_j),\,\mathrm{vec}(X_i)
\big\rangle_{_{H}}.    
\end{align*}
Applying Lemma~\ref{lemma:vec-identity} to the above yields
\begin{align*}
\langle (AA^{*}\star BB^{*})e_{j},e_{i}\rangle_{_{K}}
&=\big\langle\mathrm{vec}\big(BB^{*}X_j (A^{\#})^*A^{\#}\big), \mathrm{vec}(X_i) \big\rangle_{_{H}}\\
=\langle BB^{*}X_j (A^{\#})^*A^{\#},  X_i \rangle_{_{\Bou_2(H_1^\#,H_2)}}
&=\Tr_{_{H_2}}\left(BB^{*}X_j (A^{\#})^*A^{\#}X_i^{*}\right)\\
=\Tr_{_{H_1^{\#}}}\left(X_i^{*}BB^{*}X_j (A^{\#})^*A^{\#}
\right) 
&=\Tr_{_{\widetilde{\Hil}^{\#}}}\left(
A^{\#}X_i^{*}BB^{*}X_j (A^{\#})^*
\right)\\
=\Tr_{_{\widetilde{\Hil}^{\#}}}\left(
\big(A^{\#}X_i^{*}B\big)\; \big(B^{*}X_j (A^{\#})^*\big)
\right)
&= \left \langle B^{*}X_j (A^{\#})^*,B^{*}X_i (A^{\#})^* \right\rangle_{_{\Bou_2(\widetilde{\Hil}^\#)}},
\end{align*}
where we use the cyclical property of trace, as all the operators involved above are Hilbert--Schmidt (see {\cite[\S 18]{conway2000course}} and \cite[p. 141]{kadison1997fundamentals}). Therefore we have
\begin{align}\label{eq:Gram}
\langle (AA^*\star BB^*)e_j,e_i\rangle_{_{K}}
=\langle G_j, G_i\rangle_{_{\Bou_2(\widetilde{\Hil}^{\#})}} \qquad(i,j\in I),
\end{align}
where $G_j:=B^{*}X_j (A^{\#})^* \in \Bou_2(\widetilde{\Hil}^\#)$.\medskip

\textit{Step 2. Rank considerations.} If $\min(\rk AA^*, \rk BB^*, {\bf r}(\star)) = \infty$, then there is nothing to prove. So suppose $\min(\rk AA^*, \rk BB^*, {\bf r}(\star))<\infty$.\medskip

For finitely supported $x= \sum_{i\in I}x_i e_i \in K$, using \eqref{eq:Gram} we have
\begin{align*}
\langle (AA^*\star BB^*)x,x\rangle_{_{K}}
&= \sum_{i,j\in I}\overline{x_i}\,x_j\,\langle (AA^*\star BB^*)e_j,e_i\rangle_{_{K}}\\
= \sum_{i,j\in I}\overline{x_i}\,x_j\,\langle G_j,G_i\rangle_{_{\Bou_2(\widetilde{\Hil}^\#)}} 
&= \Big\langle \sum_{j\in I} x_j G_j,\ \sum_{i\in I} x_i G_i \Big\rangle_{_{\Bou_2(\widetilde{\Hil}^\#)}}= \|G(x)\|_{\Bou_2(\widetilde{\Hil}^{\#})}^2,
\end{align*}
where, for $X(x):=\sum_{j\in I}x_j X_j$, we let
\begin{align*}
G(x):=\sum_{j\in I} x_j G_j = B^* \Big( \sum_{j\in I}x_j X_j\Big) (A^\#)^* = B^* X(x) (A^\#)^* \in \Bou_2(\widetilde{\Hil}^{\#}).
\end{align*}
Notice that:
\begin{align*}
\rk G(x) 
= \rk B^* X(x) (A^\#)^* 
&\leq \min(\rk B^*, \rk X(x), \rk (A^\#)^*) \\
= \min(\rk BB^*, \rk X(x), \rk AA^*)
&\leq \min(\rk BB^*, {\bf r}(\star), \rk AA^*) <\infty.
\end{align*}
Therefore $G(x)$ has finite rank $r(x) := \rk G(x) \leq \min(\rk BB^*, {\bf r}(\star), \rk AA^*)$.\medskip

\textit{Step 3. The trace-rank inequality.} Let $G=G(x)\in \mathcal B_2(\widetilde{\Hil}^\#)$ have finite rank $r=r(x)$. Choose an orthonormal basis $\{h_1,\dots,h_r\}$ of $(\ker G)^\perp$ and extend
it to an orthonormal basis $\{h_j\}_{j\in J}$ of $\widetilde{\Hil}^\#$.
Then $Gh_j=0$ for $j \in J\setminus \{1,\dots,r\}$, hence
\[
\Tr(G)=\sum_{j\in J }\langle Gh_j,h_j\rangle
      =\sum_{j=1}^r \langle Gh_j,h_j\rangle.
\]
Thus, by Cauchy--Schwarz in $\C^r$ and then in $\widetilde{\Hil}^{\#}$,
\begin{align*}
|\Tr(G)|^2
&=\left|\sum_{j=1}^r \langle Gh_j,h_j\rangle\right|^2
\le r\sum_{j=1}^r |\langle Gh_j,h_j\rangle|^2\\
&\le r\sum_{j=1}^r \|Gh_j\|^2
\le r\sum_{j\in J} \|Gh_j\|^2
= r\|G\|_{_{\Bou_2(\widetilde{\Hil}^{\#})}}^2.
\end{align*}
We will use this inequality in the following steps.\medskip

\textit{Step 4. The norm bound.} We will prove that $\rho\in K$. Indeed, suppose $F\subset K$ is finite, and compute
\begin{align*}
\Big\lVert \sum_{i\in F} \langle \mathrm{vec}(BA^\#),u_i\rangle_{_{H}}e_i \Big\rVert^2 
&= \sum_{i\in F} \lvert \langle \mathrm{vec}(BA^\#),u_i\rangle_{_{H}} \rvert^2 \\ 
\leq  \sum_{i\in F} \lVert \mathrm{vec}(BA^\#) \rVert_{_{H}}^2\, \lVert u_i \rVert_{_{H}}^2
&\leq \lVert BA^\# \rVert^2_{_{\Bou_2(H_1^\#,H_2)}} \sum_{i\in I} \lVert u_i \rVert^2_{_{H}} \\
&= \lVert BA^\# \rVert^2_{_{\Bou_2(H_1^\#,H_2)}} \lVert w \rVert^2_{_{\Hil}}.
\end{align*}
Thus, $\rho\in K$, and $\lVert BA^\# \rVert_{_{\Bou_2(H_1^\#,H_2)}}\leq \lVert A \rVert_{_{\Bou_2(\widetilde{\Hil},H_1)}} \lVert B\rVert_{_{\Bou_2(\widetilde{\Hil}^\#,H_2)}}$ gives the desired norm bound on $\rho$.
\medskip

\textit{Step 5. The desired inequality.} From the previous steps we have
\begin{align*}
    |\Tr (G(x))|^2 \leq \min(\rk BB^*, {\bf r}(\star), \rk AA^*) \, \lVert G(x)\rVert^{2}_{\Bou_2(\widetilde{\Hil}^\#)}.
\end{align*}
Recall that $G_j=B^{*}X_j (A^{\#})^* \in \Bou_2(\widetilde{\Hil}^\#)$, and its trace is given by
\begin{align*}
    \Tr_{_{\widetilde{\Hil}^\#}}(G_j) 
    &= \Tr_{_{\widetilde{\Hil}^\#}}(B^{*}X_j (A^{\#})^*)
    = \Tr_{_{H_2}}(X_j (BA^{\#})^*) \\
    = \langle X_j, BA^\#\rangle_{_{\Bou_2(H_1^{\#},H_2)}}
    &= \overline{\langle BA^\#, X_j\rangle}_{_{\Bou_2(H_1^{\#},H_2)}}
    = \overline{\langle \mathrm{vec}(BA^\#), u_j\rangle}_{_{H}}
\end{align*}
Therefore,
\begin{align*}
\left| \Tr (G(x)) \right|^2 
&= \left| \sum_{j\in I} x_j \Tr(G_j) \right|^2 \\
&= \left| \sum_{j\in I} x_j \overline{\langle \mathrm{vec}(BA^\#), u_j\rangle}_{_{H}}\right|^2 
= \left| \langle x, \rho\rangle_{_{K}} \right|^2\\
&= \langle x, \rho\rangle_{_{K}} \langle \rho, x\rangle_{_{K}} =  \Big\langle \langle x, \rho\rangle_{_{K}}\rho, x\Big \rangle_{_{K}} 
=  \langle \ro_{\rho,\rho}x, x  \rangle_{_{K}}.
\end{align*}
Thus, using the computation in Step 2 and the inequality at the beginning of this step, we have shown
\[
\Big\langle \big(AA^*\star BB^*-\frac{1}{\min(\rk AA^*, \rk BB^*, {\bf r}(\star))}\ro_{\rho,\rho}\big) x,x\Big\rangle_K \geq 0
\]
for all finitely supported $x\in K$. Since such vectors are dense in $K$ and the quadratic forms are continuous, it extends to all $x\in K$, as desired.\medskip

{\bf Part II. Optimality.} We break this into two cases.\medskip

\textit{Case 1.} Suppose ${\bf r}(\star)<\infty$. Then the supremum is attained, and so there exists $X_0\in \mathcal{U}(\star)$ such that $\rk X_0={\bf r}(\star)$. Thus, from the singular value decomposition of $X_0$, there exist orthonormal $\{\overline{x_1},\dots,\overline{x_{{\bf r}(\star)}}\}\subset H_1^\#$ and $\{{y_1},\dots, {y_{{\bf r}(\star)}}\}\subset H_2$, and positive real numbers $\{\sigma_1,\dots,\sigma_{{\bf r}(\star)}\}$ such that
\[
X_0 = \sum_{j=1}^{{\bf r}(\star)} \sigma_j \ro_{y_j,\overline{x_j}} = \sum_{j=1}^{r} \sigma_j \ro_{y_j,\overline{x_j}} + \sum_{j=r+1}^{{\bf r}(\star)} \sigma_j \ro_{y_j,\overline{x_j}},
\]
where $r:= \min\{\dim{H_1},\,\dim H_2,\, \dim \widetilde{\Hil},\, {\bf r}(\star)\}$. Suppose $\{z_1,\dots,z_r\}\subset \widetilde{\Hil}$ is an orthonormal set, and define $A: \widetilde{\Hil}\to H_1$ and $B:\widetilde{\Hil}^{\#}\to H_2$ by:
\[
A = \sum_{j=1}^{r} \ro_{x_j,z_j}\qquad\mbox{and}\qquad  B = \sum_{j=1}^{r} \sigma_j^{-1} \ro_{y_j,\overline{z_j}}.
\]
Therefore $\rk AA^* = \rk BB^* = r$ and so $\min(\rk AA^*, \rk BB^*, {\bf r}(\star))=r$. Since
\[
B^*= \sum_{j=1}^{r} \sigma_j^{-1} \ro_{\overline{z_j},y_j}, \qquad 
(A^\#)^* = \sum_{j=1}^{r} \ro_{\overline{x_j},\overline{z_j}}, 
\]
it follows that 
\[
B^* X_0 (A^\#)^* = \sum_{j=1}^{r} \ro_{\overline{z_j},\overline{z_j}}.
\]
Thus $\Tr (B^* X_0 (A^\#)^*) = r = \lVert B^* X_0 (A^\#)^* \rVert_2^2 = \rk (B^* X_0 (A^\#)^*)$, and so the trace-rank inequality from Step 5 of Part 1 holds with equality. In other words, there exists $x_0\in K$, and $A$ and $B$ (as above) such that $\min (\rk AA^*, \rk BB^*, {\bf r}(\star))=r$ and
\[
\Big\langle \Big(AA^* \star BB^* - \frac{1}{r} \ro_{\rho,\rho}\Big) x_0, x_0 \Big\rangle_{_{K}} = 0.
\]
Thus, for any $\epsilon>0$,
\[
\Big\langle \Big(AA^* \star BB^* - (1/r+\epsilon) \ro_{\rho,\rho}\Big) x_0, x_0 \Big\rangle_{_{K}} < 0.
\]
Therefore, one cannot improve the coefficient $1/\min (\rk AA^*, \rk BB^*, {\bf r}(\star))$ uniformly over all $A,B$.\medskip

\textit{Case 2.} Suppose ${\bf r}(\star)=\infty$. Then, similar to the previous case, suppose $X_0\in \mathcal{U}(\star)$ with singular value decomposition
\[
X_0 = \sum_{j=1}^{\infty} \sigma_j \ro_{y_j,\overline{x_j}} = \sum_{j=1}^{r} \sigma_j \ro_{y_j,\overline{x_j}} + \sum_{j=r+1}^{\infty} \sigma_j \ro_{y_j,\overline{x_j}},
\]
where $r=\min\{\dim{H_1},\dim H_2, \dim \widetilde{\Hil}, {\bf r}(\star)\}$ if the right hand side is finite, otherwise take $r$ to be any arbitrary positive. Then follow the same steps constructing $A$ and $B$ as in the above case.
\end{proof}

\subsection{Proof of Theorem~\ref{T:Hilbert-tensor}}

\begin{proof}[Proof of Theorem~\ref{T:Hilbert-tensor}] There is nothing to prove if $\min(\rk AA^*, \rk BB^*)=\infty$ so suppose it is finite. We follow the construction and notation in Theorem~\ref{T:thmC-for-rank1-fixed}. Suppose $\{e_\alpha\}_{\alpha\in\A}$ and $\{f_\beta\}_{\beta\in \B}$ are orthonormal bases of $H_1$ and $H_2$, respectively. We choose
\begin{align*}
K=H_1\otimes H_2 \qquad \mbox{and}\qquad  \{e_{\alpha}\otimes f_\beta\}_{(\alpha,\beta)\in \A\times \B}
\end{align*}
as the orthonormal basis of $K$. We divide the proof into two parts.\medskip

\textbf{Part I. $\max(\rk AA^*,\rk BB^*)<\infty$.} Then there exists finite $\A'\subset \A$ with $\Ran (AA^*) = \mathrm{span} \{e_{\alpha}\}_{\alpha\in \A'}$ and finite $\B'\subset\B$ with $\Ran (BB^*) = \mathrm{span}\{f_{\beta}\}_{\beta\in \B'}$.\medskip

\textit{Step 1. The admissible vector.} Letting $\1$ denote the indicator function, consider the vector
\begin{align*}
w:=\sum_{(\alpha,\beta)\in \A\times \B} (e_\alpha\otimes f_\beta) \otimes (\1_{_{(\alpha,\beta) \in \A'\times\B'}}e_\alpha\otimes f_\beta).
\end{align*}
Since both $\A'$ and $\B'$ are finite,  $w\in \Hil:=K\otimes(H_1\otimes H_2)$. Moreover, the vectors $\{u_i\}_{i\in I}$ in Theorem~\ref{T:thmC-for-rank1-fixed} become
\[
\big\{\1_{_{(\alpha,\beta) \in \A'\times\B'}}e_\alpha\otimes f_\beta\big\}_{\alpha\in \A,\beta\in \B}.
\]

\textit{Step 2. The $\star$ product and the inequality.} We choose the $\star$ product
\[
\star
=\star(\ro_{w,w})\in \prd(H_1,H_2;H_1\otimes H_2;\{e_{\alpha}\otimes f_\beta\}_{(\alpha,\beta)\in \A\times\B}).
\]
Since $AA^*$ is zero outside the span of $\{e_{\alpha}\}_{\alpha\in \A'}$ and $BB^*$ is zero outside the span of $\{f_{\beta}\}_{\beta\in\B'}$, using Lemma~\ref{prop:rank-one-slicing}, we have
\[
AA^* \star BB^* = AA^* \otimes BB^*.
\]
Thus, from Theorem~\ref{T:thmC-for-rank1-fixed}, we have 
\begin{align*}
AA^* \otimes BB^* \;  \succcurlyeq \; \frac{1}{\min\big(\rk AA^*,\rk BB^*,{\bf r}(\star)\big )} \ro_{\rho,\rho} \;  \succcurlyeq \; 0,
\end{align*}
where
\begin{align*}
    \rho &\; := \; \sum_{(\alpha,\beta)\in \A\times \B} \; \langle \, \mathrm{vec}(BA^\#),\1_{_{(\alpha,\beta)\in \A'\times \B'}} e_\alpha\otimes f_\beta\,\rangle_{_{H_1\otimes H_2}}\; e_\alpha\otimes f_\beta \; \in \; K.
\end{align*}

\textit{Step 3. Showing $\rho=\mathrm{vec}(BA^\#)$.} Let $P_1$ be the orthogonal projection of $H_1$ onto $\mathrm{span}\{e_{\alpha}\}_{\alpha\in \A'}$ and $P_2$ be the orthogonal projection on $H_2$ onto $\mathrm{span}\{f_{\beta}\}_{\beta\in \B'}$. Since $\Ran (AA^*) = \Ran A = \mathrm{span}\{e_{\alpha}\}_{\alpha\in \A'}$, we have $P_1A=A$. From Lemma~\ref{lemma:transposition}, we get $A^\# = (P_1 A)^\# = A^\# P_1^\#$. Similarly, $\Ran(BB^*)=\Ran(B)=\mathrm{span}\{f_{\beta}\}_{\beta\in \B'}$ implies $P_2B=B$, hence
\[
BA^\# = (P_2B)(A^\#P_1^\#)= P_2\,(BA^\#)\,P_1^\#.
\]
Consequently,
\[
(BA^\#)\,\overline{e_\alpha}=0 \quad (\alpha\notin\A'),
\qquad
\langle (BA^\#)\,\overline{e_\alpha},\, f_\beta\rangle=0 \quad (\beta\notin\B').
\]
Using Lemma~\ref{lemma:vec-identity}, we have
\[
\langle \mathrm{vec}(BA^\#),\, e_\alpha\otimes f_\beta\rangle_{_{H_1\otimes H_2}}
= \big\langle (BA^\#)\,\overline{e_\alpha},\, f_\beta\big\rangle_{_{H_2}},
\]
and so we deduce that
\[
\langle \mathrm{vec}(BA^\#),\, e_\alpha\otimes f_\beta \rangle_{_{H_1\otimes H_2}} = 0
\qquad\text{provided }(\alpha,\beta)\notin \A'\times\B'.
\]
Therefore $\mathrm{vec}(BA^\#)\in \mathrm{span}\{e_\alpha\otimes f_\beta:(\alpha,\beta)\in \A'\times\B'\}$, and hence
\begin{align*}
\rho
&=\sum_{(\alpha,\beta)\in \A\times \B}\1_{_{(\alpha,\beta)\in\A'\times\B'}}
\langle \mathrm{vec}(BA^\#),\,e_\alpha\otimes f_\beta \rangle_{_{H_1\otimes H_2}}\, e_\alpha\otimes f_\beta\\
&=\sum_{(\alpha,\beta)\in \A\times \B}
\langle \mathrm{vec}(BA^\#),\,e_\alpha\otimes f_\beta\rangle_{_{H_1\otimes H_2}}\, e_\alpha\otimes f_\beta
=\mathrm{vec}(BA^\#),
\end{align*}
where the last equality is simply the expansion of $\mathrm{vec}(BA^\#)$ in the orthonormal basis
$\{e_\alpha\otimes f_\beta\}_{\alpha,\beta}$, and terms outside $\A'\times\B'$ are already $0$.\medskip

\textit{Step 4. Showing ${\bf r}(\star)\geq \rk BB^*$.} Notice that 
\[
\mathcal{U}(\star)\; :=\; \overline{\mathrm{span}}^{\Bou_2}\,\big\{\mathrm{vec}^{-1}(\1_{_{(\alpha,\beta) \in \A'\times\B'}}e_\alpha\otimes f_\beta):{(\alpha,\beta)\in \A\times \B}\big\},
\]
contains rank-one operators $\ro_{f_{\beta},\overline{e_{\alpha}}}$ for $\alpha\in \A'$ and $\beta\in \B'$, showing the final desired inequality in this case.\medskip

\textbf{Part II.  $\rk BB^* <\rk AA^*=\infty$.} Let $A_n \to A$ in $\Bou_2(\widetilde{\Hil},H_1)$ be finite rank approximations via, e.g., the singular value expansion of $A$, such that $\rk A_nA_n^* > \rk BB^*$. Then,
\[
\min(\rk A_nA_n^* ,\rk BB^*)=\min(\rk AA^*,\rk BB^*),
\]
and from the previous part
\[
A_nA_n^* \otimes BB^* \;  \succcurlyeq \; \frac{1}{\min\big(\rk AA^*,\rk BB^* \big )} \ro_{\mathrm{vec}(BA_n^\#),\mathrm{vec}(BA_n^\#)} \;  \succcurlyeq \; 0.
\]

\textit{Step 1. $A_nA_n^*\to AA^*$ in trace norm.} Using the ideal property $\|XY\|_1\le \|X\|_2\|Y\|_2$ for Hilbert--Schmidt operators {\cite[Theorem~7.8]{weidmann1980linear}}, we get
\begin{align*}
\|AA^*-A_nA_n^*\|_1
&=\|(A-A_n)A^*+A_n(A-A_n)^*\|_1\\
&\le \|A-A_n\|_2\|A\|_2+\|A_n\|_2\|A-A_n\|_2.    
\end{align*}
Since $\|A_n\|_2\to \|A\|_2$ and $\|A-A_n\|_2\to 0$, it follows that $\|AA^*-A_nA_n^*\|_1\to0$, hence also $\|AA^*-A_nA_n^*\|\to0$ in operator norm.\medskip

\textit{Step 2. The tensor products converge.} Because $BB^*$ is trace-class and $\|X\otimes Y\| = \|X\|\,\|Y\|$ {\cite[p. 146]{kadison1997fundamentals}},
\[
\|(AA^*-A_nA_n^*)\otimes BB^*\|
= \|AA^*-A_nA_n^*\|\,\|BB^*\|\ \longrightarrow\ 0.
\]
Hence $A_nA_n^*\otimes BB^*\to AA^*\otimes BB^*$ in the operator norm.\medskip

\textit{Step 3. Convergence of rank-one terms.} First note that $BA_n^\#\to BA^\#$ in trace norm, hence in Hilbert--Schmidt norm:
\[
\|BA_n^\#-BA^\#\|_1 \le \|B\|_2\,\|A_n^\#-A^\#\|_2=\|B\|_2\,\|A_n-A\|_2\to 0,
\]
and $\|T\|_2\le \|T\|_1$ for trace-class $T$. Since $\mathrm{vec}:\Bou_2(H_1^\#,H_2)\to H_1\otimes H_2$ is an isometry, this implies
$\mathrm{vec}(BA_n^\#)\to \mathrm{vec}(BA^\#)$ in $H_1\otimes H_2$. Set $v_n:=\mathrm{vec}(BA_n^\#)$ and $v:=\mathrm{vec}(BA^\#)$. For rank-one operators one has $\|\ro_{x,y}\|=\|x\|\,\|y\|$, hence
\begin{align*}
\|\ro_{v_n,v_n}-\ro_{v,v}\|
&=\|\ro_{v_n,v_n}+\ro_{-v,v_n}+\ro_{v,v_n}+\ro_{v,-v}\| \\
=\|\ro_{v_n-v,v_n}+\ro_{v,v_n-v}\| 
&\le \|v_n-v\|\,\|v_n\|+\|v\|\,\|v_n-v\|.    
\end{align*}
Since $v_n\to v$, the right-hand side tends to $0$, so
$\ro_{v_n,v_n}\to \ro_{v,v}$ in operator norm. Since the left-hand side of the desired inequality also converges in operator norm and the positive cone is closed under operator-norm limits {\cite[Proposition 3.5]{conway2000course}}, the final Loewner inequality holds.
\end{proof}

\subsection{Proof of Theorem~\ref{T:Schur-ext-Hil}}

We begin with a lemma.

\begin{lemma}\label{lemma:trace-case}
With notation in Theorem~\ref{T:Schur-ext-Hil}, consider a positive operator $T\in \Bou_1(\Hil)$ with spectral resolution in the trace-norm: $
T=\sum_{n=1}^\infty \lambda_n \, \ro_{w_n, w_n}$. Suppose $\star_n=\star_n(\ro_{w_n, w_n})$ for $n\geq 1$, and $\star = \star (T)$. Then 
\[
P\star Q=\sum_{n=1}^{\infty} \lambda_n( P \star_n Q) \qquad (P\in \Bou_2(H_1),~Q\in \Bou_2(H_2)),
\]
with convergence in trace norm, with bound $\lVert P\star Q \rVert_1 \leq  \|T\|_1\,\|P\|_2\,\|Q\|_2 $.
\end{lemma}
\begin{proof} Letting $T^{(n)}:=\ro_{w_n,w_n}$ and using Proposition~\ref{prop:slicing+stitching},  we have $ T_{e_i,e_j}= \sum_{n=1}^{\infty}\lambda_n \, T^{(n)}_{e_i,e_j}$, where convergence is in $\Bou_1(H)$ for each $i,j \in I$; $\lambda_n\ge 0$, $\sum_{n=1}^\infty\lambda_n<\infty$, and $(w_n)_{n\ge1}$ is an orthonormal family in $\Hil$. For $p\geq 1$:
\begin{align*}
\star_{S^{(p)}} := \star (S^{(p)}) \qquad \mbox{where}\qquad S^{(p)}:=\sum_{n=1}^{p} \lambda_n \, T^{(n)}.   
\end{align*}
By Proposition~\ref{prop:slicing+stitching}, we have $P\star_{S^{(p)}}Q = \sum_{n=1}^{p} \lambda_n (P\star_n Q)$. We break the rest of the proof into two steps.\medskip

\textit{Step 1. Showing $\sum_{n=1}^{\infty} \lambda_n( P \star_n Q) \in \Bou_1(K)$.} Using Lemma~\ref{prop:rank-one-slicing} for the norm-bounds, for positive integers $p<q$, using the triangle inequality for the trace-norm, we have
\begin{align*}
\lVert P\star_{S^{(q)}}Q - P\star_{S^{(p)}}Q \rVert_1 
&\leq  \sum_{n=p+1}^q \big\|\lambda_n (P\star_n Q)\big\|_1
= \sum_{n=p+1}^q \lambda_n\,\|P\star_n Q\|_1 \\
\le \sum_{n=p+1}^q \lambda_n\,\|P\|_2\,\|Q\|_2 
&\leq  \|P\|_2\,\|Q\|_2 \sum_{n=p+1}^q \lambda_n \to 0 \mbox{ as }p<q\to \infty.
\end{align*}
It follows that the partial sums $P\star_{S^{(p)}}Q$ form a Cauchy sequence in $\mathcal B_1(K)$, and so, as $(\mathcal B_1(K),\|\cdot\|_1)$ is a Banach space, they converge in the trace-norm to some $S\in \mathcal B_1(K)$. In other words, the series
\[
S:=\sum_{n=1}^\infty \lambda_n (P\star_n Q)
\]
defines a trace-class operator on $K$. \medskip

\textit{Step 2. Identifying $S$ and $P\star Q$.} Fix $i,j\in I$, and note that by definition of the products and by Proposition~\ref{prop:slicing+stitching} we have,
\begin{align*}
\langle (P\star Q)e_j, e_i\rangle_{_{K}}
  = \Tr\big(T_{e_i,e_j}^*(P\otimes Q)\big)
  &= \Tr\Big(\Big(\sum_{n=1}^\infty \lambda_n T^{(n)}_{e_i,e_j}\Big)^*(P\otimes Q)\Big)\\
  = \Tr\Big(\sum_{n=1}^\infty \lambda_n (T^{(n)}_{e_i,e_j})^*(P\otimes Q)\Big)
&= \sum_{n=1}^\infty \lambda_n\,\Tr\big((T^{(n)}_{e_i,e_j})^*(P\otimes Q)\big).    
\end{align*}
On the other hand we also have,
\begin{align*}
\langle Se_j,e_i\rangle_{_{K}}
 &= \Big\langle \sum_{n=1}^\infty \lambda_n (P\star_n Q)e_j,e_i\Big\rangle_{_{K}} = \sum_{n=1}^\infty \lambda_n\,\langle(P\star_n Q)e_j,e_i\rangle_{_{K}} \\
 &= \sum_{n=1}^\infty \lambda_n\,\Tr\big((T^{(n)}_{e_i,e_j})^*(P\otimes Q)\big).
\end{align*}
Therefore, for every $i,j\in I$, $\langle (P\star Q)e_j,e_i\rangle_{_{K}} = \langle Se_j,e_i\rangle_{_{K}}$. This implies $
P\star Q = S = \sum_{n=1}^\infty \lambda_n (P\star_n Q)$ in $\mathcal B_1(K)$. Finally we have
\begin{align*}
\lVert P\star Q \rVert_1 
&\leq \sum_{n=1}^\infty \|\lambda_n (P\star_n Q)\|_1
\le \|P\|_2\,\|Q\|_2 \sum_{n=1}^\infty \lambda_n\\
&= \|T\|_1\,\|P\|_2\,\|Q\|_2 
<\infty.    
\end{align*}
This shows the desired norm bound.
\end{proof}

\begin{proof}[Proof of Theorem~\ref{T:Schur-ext-Hil}] We divide the proof into four parts.\medskip

\textbf{Part I. Norm bounds.} The case of $T \in \Bou_1(\Hil)$ follows from Lemma~\ref{lemma:trace-case}. We consider the $T \in \Bou_2(\Hil)$ case. Consider the product $\star=\star(T)\in \prd(2)$. Let $\{f_\alpha\}_{\alpha}$ be an orthonormal basis of $H$. Then the family $\{e_i\otimes f_\alpha\}_{i\in I,\alpha}$ is an orthonormal
basis of $\Hil$. By the definition of the slices $T_{ij}:=T_{e_i,e_j}$, we have
\[
\langle T_{ij} f_\beta,\,f_\alpha\rangle_{_{H}}
 = \langle T(e_j\otimes f_\beta),\, e_i\otimes f_\alpha\rangle_{_{\Hil}}.
\]
From Proposition~\ref{prop:slicing+stitching}, each $T_{ij}\in \Bou_2(H)$, and so we write
\[
\|T_{ij}\|_2^2
 = \sum_{\alpha,\beta} \big|\langle T_{ij}f_\beta,f_\alpha\rangle\big|^2
 = \sum_{\alpha,\beta} \big|\langle                      T(e_j\otimes f_\beta),e_i\otimes f_\alpha\rangle\big|^2.
\]
Using the above, since $T\in \Bou_2(\Hil)$, we also have
\begin{align*}
\|T\|_2^2 =\sum_{i,j}\sum_{\alpha,\beta}
      \big|\langle T(e_j\otimes f_\beta),e_i\otimes f_\alpha\rangle\big|^2 = \sum_{i,j\in I} \|T_{ij}\|_2^2.
\end{align*}
Now, for each $i,j\in I$, by the Cauchy--Schwarz inequality we have
\[
\big|\langle P\star Q e_j, e_i\rangle_{_{K}}\big|^2 = |\langle P\otimes Q,T_{ij}\rangle_{_{\Bou_2(H)}}|^2
 \le \|T_{ij}\|_2^2\,\|P\otimes Q\|_2^2.
\]
Summing over $i,j\in I$, we obtain
\begin{align*}
\sum_{i,j\in I} \big|\langle P\star Qe_j,e_i\rangle_{_{K}}\big|^2
&\le \sum_{i,j\in I} \|T_{ij}\|_2^2\,\|P\otimes Q\|_2^2 \\
 = \|P\otimes Q\|_2^2 \sum_{i,j\in I} \|T_{ij}\|_2^2
 &= \|P\otimes Q\|_2^2\,\|T\|_2^2.
\end{align*}
We thus conclude that $P\star Q\in \Bou_2(K)$ and $\|P\star Q\|_2 \le \|T\|_2\,\|P\|_2\,\|Q\|_2$. \medskip

\textbf{Part II. Positivity.} Fix $\star = \star(T)\in \prd(2)$. Suppose that $P\; \succcurlyeq \;0$ and $Q\; \succcurlyeq \;0$. We will show that $P\star Q\; \succcurlyeq \;0$. Let $x=\sum_{i\in I}x_i e_i\in K$ with finite support. Then
\begin{align*}
\langle (P\star Q)x, x\rangle_{_{K}}
&= \sum_{i,j\in I}\overline{x_i}x_j \,\langle (P\star Q)e_j,e_i\rangle_{_{K}} \\
= \sum_{i,j\in I}\overline{x_i}x_j \,\Tr\big(T_{ij}^*(P\otimes Q)\big) 
&= \Tr\Big(\Big(\sum_{i,j\in I} \overline{x_i} x_j\,T_{ij}\Big)^*(P\otimes Q)\Big).
\end{align*}
For $\xi,\eta\in H$, we compute
\begin{align*}
\langle T_{x,x}\eta, \xi \rangle_{_H}
=\langle T (x\otimes\eta), x\otimes \xi \rangle_{_{\Hil}}
&= \Big\langle T \Big(\sum_{j\in I} x_j e_j \otimes \eta \Big), \sum_{i\in I} x_i e_i \otimes \xi \Big\rangle_{_{\Hil}} \\
= \sum_{i,j\in I} \overline{x_i}\,x_j\,\langle T(e_j\otimes\eta),e_i\otimes\xi\rangle_{_{\Hil}} 
&= \sum_{i,j\in I} \overline{x_i}\,x_j\,\langle T_{ij}\eta,\xi\rangle_{_H}\\
&= \Big\langle \Big(\sum_{i,j} \overline{x_i} x_j T_{ij}\Big)\eta ,\xi\Big\rangle_{_H}.
\end{align*}
Therefore $\sum_{i,j\in I}x_j \overline{x_i}\,T_{ij}$ is precisely the slice $T_{x,x}$. Moreover, it is immediate that $T \; \succcurlyeq \;0$ implies $T_{x,x} \; \succcurlyeq \;0$. Because $P\otimes Q \; \succcurlyeq \;0$, and the trace of a product of two positive operators is nonnegative, we get
\[
\langle (P\star Q)x,x\rangle_{_{K}}
 = \Tr(T_{x,x}^{1/2}(P\otimes Q)T_{x,x}^{1/2}) \;\ge\; 0,
\]
where $T_{x,x}^{1/2}$ is the positive square root of $T_{x,x}$, which uniquely exists {\cite[Proposition 3.3]{conway2000course}}. By density of finite-support vectors in $K$, $P\star Q \; \succcurlyeq \; 0$.\medskip

\textbf{Part III. The lower bound is trace-class.} We now show that 
\[
\sum_{n=1}^\infty \frac{\lambda_n}{m(n)}\,\ro_{\rho_n,\rho_n} \in \Bou_1(K),
\] 
where $m(n) := \min(\rk AA^*, \rk BB^*, {\bf r}(n))$. Since each $\ro_{\rho_n,\rho_n}$ is a rank-one positive operator, its trace norm equals its trace:
\begin{align*}
\big\| \ro_{\rho_n,\rho_n}\big\|_1
= \Tr(\ro_{\rho_n,\rho_n})
= \|\rho_n\|^2.    
\end{align*}
Thus, 
\[
\Big\|\frac{\lambda_n}{m(n)}\,\ro_{\rho_n,\rho_n}\Big\|_1
= \frac{\lambda_n}{m(n)}\,\|\rho_n\|^2.
\]
Since $m(n)\ge 1$ whenever $m(n)\in (0,\infty)$, we have $\frac{\lambda_n}{m(n)}\,\|\rho_n\|^2
 \le \lambda_n\,\|\rho_n\|^2$. Thus it suffices to show that $\sum_{n=1}^\infty \lambda_n\,\|\rho_n\|^2 <\infty$. By Theorem~\ref{T:thmC-for-rank1-fixed}, 
\[
\|\rho_n\|
 \le \|A\|_2\,\|B\|_2\,\sqrt{\|\ro_{w_n, w_n}\|_2}= \|A\|_2\,\|B\|_2\,\|w_n\|
 = \|A\|_2\,\|B\|_2.
\]
Hence $\sum_{n=1}^\infty \lambda_n\,\|\rho_n\|^2 \le \|A\|_2^2\,\|B\|_2^2 \sum_{n=1}^\infty \lambda_n<\infty$.\medskip

\textbf{Part IV. The limiting inequality.} For each integer $p\geq 1$ consider the following for $P=AA^*$ and $Q=BB^*$:
\[
\mathcal{X}_p := \sum_{n=1}^p \lambda_n (P\star_n Q)
\qquad \mbox{and}\qquad
\mathcal{Y}_p := \sum_{n=1}^p \frac{\lambda_n}{m(n)}\,\ro_{\rho_n,\rho_n}.
\]
We have shown that $\mathcal{X}_p\to P\star Q$ and $\mathcal{Y}_p\to \sum_{n\geq 1} \frac{\lambda_n}{m(n)} \ro_{\rho_n,\rho_n}$ in the trace norm, and hence in the operator norm. Moreover, from Theorem~\ref{T:thmC-for-rank1-fixed}, for each $n$,
$\lambda_n(P\star_n Q) \;\succcurlyeq\;
\frac{\lambda_n}{m(n)}\,\ro_{\rho_n,\rho_n},$ so summing over $n=1,\dots,p$ gives $\mathcal{X}_p \;\succcurlyeq\; \mathcal{Y}_p$ for all $p\geq 1$. Since $\mathcal{X}_p - \mathcal{Y}_p$ converges to $P\star Q - \sum_{n\geq 1} \frac{\lambda_n}{m(n)} \ro_{\rho_n,\rho_n}$ in the operator norm, and the positive cone is closed under operator norm limits {\cite[Proposition 3.5]{conway2000course}}, the desired inequality holds.
\end{proof}

\section{Canonical formulation}\label{S:canonical-form}

In the preceding sections, the $\star$ products and the associated lower bounds were developed using a fixed orthonormal basis of the Hilbert space $K$. We now show that these constructions for rank-one admissible operators admit a canonical formulation, and the lower bound is independent of the choice of the orthonormal basis.

As before, we let $H_1,H_2,K$ be Hilbert spaces, and set $H:=H_1\otimes H_2$ and $\Hil:=K\otimes H$. For the conjugate linear isometries $\conj_K:K \to K^\#$ and $\conj_H:H\to H^\#$, recall that $\conj_K\otimes \conj_H$ is the unique conjugate linear map $:K\otimes H \to K^\#\otimes H^\#$ defined on simple tensors $k\otimes h \in K\otimes H$ via 
\[
(\conj_K\otimes \conj_H) (k\otimes h):=(\conj_K k)\otimes (\conj_Hh) =\overline{k} \otimes \overline{h}.
\]
Using the isometry $\conj_K\otimes \conj_H$ we construct an operator $\W:K^\#\longrightarrow H^\#$ that will play a crucial role in the development below. 

\begin{defn}\label{defn:canonical}
Let $H_1,H_2,K$ be Hilbert spaces, and set $H:=H_1\otimes H_2$ and $\Hil:=K\otimes H$. Given $w\in \Hil$, we define $\W:K^\#\longrightarrow H^\#$ via 
\begin{align*}
\langle \W \overline{\xi},\ \overline{h}\rangle_{_{H^{\#}}} 
\;:=\; \Big\langle  (\conj_{K}\otimes \conj_{H})\, w,\ \overline{\xi}\otimes \overline{h}\Big\rangle_{_{K^\# \otimes H^{\#}}}
\end{align*}
for all $\xi\in K$ and $h\in H$.
\end{defn}

\begin{lemma}\label{lemma:canonical}
The map $\W$ in Definition~\ref{defn:canonical} is bounded and conjugate linear and $\conj_H^{-1}\,\W \in \Bou_2(K^\#,H)$ with $\Vert \conj_H^{-1}\W \rVert_2 = \lVert w \rVert$. Moreover, for an orthonormal basis $\{e_i\}_{i\in I}$ of $K$, and $w\in \Hil=K\otimes H$ with decomposition $w=\sum_{i\in I} e_{i}\otimes u_i$, we have that $\conj_H^{-1}\,\W \,\overline{e_i}=u_i$.
\end{lemma}
\begin{proof} Let $\lambda\in\C$ and $\xi\in K$. To avoid confusion, we let $\ast_{_{K^\#}}$ and $\ast_{_{H^\#}}$ denote the scalar multiplications in $K^\#$ and $H^\#$ respectively. In $K^\#$ one has $\overline{\lambda}\ast_{_{K^{\#}}}\overline{\xi}=\overline{{\lambda}\,\xi}$. Hence, for every $h\in H$,
\begin{align*}
\big\langle \W (\overline{\lambda}\ast_{_{K^\#}} \overline{{\xi}}),\overline{h}\big\rangle_{_{H^\#}}
=\big\langle \W (\overline{\lambda {\xi}}),\overline{h}\big\rangle_{_{H^\#}}
&=\big\langle (\conj_K \otimes \conj_H) w,(\overline{{\lambda}\,\xi})\otimes \overline{h}\big\rangle_{_{K^\# \otimes H^{\#}}} \\
=\big\langle (\conj_K \otimes \conj_H) w, (\overline{\lambda}\ast_{_{K^\#}}\overline{\xi})\otimes \overline{h}\big\rangle_{_{K^\# \otimes H^{\#}}}
&=\big\langle (\conj_K \otimes \conj_H) w, \overline{\lambda}(\overline{\xi}\otimes \overline{h})\big\rangle_{_{K^\# \otimes H^{\#}}}\\
= \lambda\langle (\conj_K \otimes \conj_H) w, \overline{\xi}\otimes \overline{h}\rangle_{_{K^\# \otimes H^{\#}}}
&= \lambda \langle \W \overline{\xi},\overline{h}\rangle_{_{H^\#}}
= \langle \lambda \ast_{_{H^\#}} \W \overline{\xi},\overline{h}\rangle_{_{H^{\#}}}.
\end{align*}
Since this holds for all $\overline{h}\in H^\#$, $\W$ is conjugate multiplicative. Additivity follows using the bi-additivity of the tensor product. We show boundedness. Fix $\xi\in K$. For all $h\in H$, by Cauchy--Schwarz in
$K^{\#}\otimes H^\#$, we have
\begin{align*}
\big|\langle \W\overline{\xi},\ \overline{h}\rangle_{_{H^\#}}\big|
&=\big|\big\langle (\conj_K\otimes \conj_H)w,\ \overline{\xi}\otimes \overline{h}\big\rangle_{_{K^{\#}\otimes H^{\#}}}\big| \\
&\le \|(\conj_K\otimes \conj_H)w\|_{_{K^{\#}\otimes H^{\#}}}\,
   \|\overline{\xi}\otimes \overline{h}\|_{_{K^{\#}\otimes H^{\#}}}.
\end{align*}
Since $\conj_K \otimes \conj_H:K\otimes H\to K^{\#}\otimes H^\#$ is an isometry,
\[
\|(\conj_K\otimes \conj_H)w\|_{_{K^{\#}\otimes H^{\#}}}=\|w\|_{_{K\otimes H}}.
\]
Moreover,
\[
\|\overline{\xi}\otimes \overline{h}\|_{_{K^{\#}\otimes H^{\#}}}
=\|\overline{\xi}\|_{_{K^{\#}}}\,\|\overline{h}\|_{_{H^{\#}}}
=\|\xi\|_{_{K}}\,\|h\|_{_{H}}.
\]
Therefore
\[
\big|\langle \W\overline{\xi},\ \overline{h}\rangle_{_{H^\#}}\big|
\le \|w\|\,\|\xi\|\,\|h\| \qquad (h\in H).
\]
Taking the supremum over all $h\in H$ with $\|h\|=1$ yields
\[
\|\W\,\overline{\xi}\|
=\sup_{\|h\|=1}\big|\langle \W\overline{\xi},\ \overline{h}\rangle_{_{H^{\#}}}\big|
\le \|w\|\,\|\xi\|.
\]
Hence $\W$ is bounded, with $\|\W\|\le \|w\|$. Write $w=\sum_{i\in I} e_i\otimes u_i$ where $\lVert w \rVert^2=\sum_{i\in I} \lVert u_i \rVert^2 <\infty$. Then for every $h\in H$,
\begin{align*}
\langle \W \, \overline{e_i},\,\overline{h}\rangle_{_{H^{\#}}}
&=\langle \conj_K\otimes \conj_H (w),\ \overline{e_i}\otimes \overline{h}\rangle_{_{K^{\#}\otimes H^{\#}}}\\
&=\Big\langle \sum_{k\in I} \overline{e_k}\otimes \overline{u_k},\ \overline{e_i}\otimes \overline{h}\Big\rangle_{_{K^{\#}\otimes H^{\#}}}
=\langle \overline{u_i},\overline{h}\rangle_{_{H^{\#}}}.
\end{align*}
Therefore, for all $i\in I$, we have $\W \, \overline{e_i}=\overline{u_i} \in H^\#$ and so 
\begin{align*}
\conj_H^{-1} \W \,\overline{e_i} = u_i \in H.
\end{align*}
Since both $\conj_H^{-1}$ (the inverse of the bijective conjugate linear $\conj_H$) and $\W$ are conjugate linear, we have that $\conj_H^{-1} \W$ is linear $:K^\# \to H$. Using this, and the square summability of $(u_i)_{i\in I}$, we obtain that $\conj_H^{-1}\W\in \Bou_2(K^{\#},H)$ and $\Vert \conj_H^{-1}\W \rVert_2 = \lVert w \rVert$.  
\end{proof}

We can now prove our main result in this section. 

\begin{theorem}\label{T:canonical}
Let $H_1,H_2,K$ be Hilbert spaces, and $H:=H_1\otimes H_2$ and $\Hil:=K\otimes H$. Suppose $T:=\ro_{w,w}\in\Bou(\Hil)$ for $w\in \Hil$. For an orthonormal basis $\{e_i\}_{i\in I}$ of $K$, define $C_e: K \to K^\#$ as the linear isometric isomorphism given by:
\begin{align*}
C_e\Big( \sum_{i\in I} \langle x, e_i\rangle_{_{K}} e_i \Big) 
:= \sum_{i\in I} \overline{\langle x, e_i\rangle_{_{K}}} e_i
= \sum_{i\in I} \langle x, e_i\rangle_{_{K}} \ast \overline{e_i}.
\end{align*}
\begin{enumerate}[(1)]
\item For $P\in\Bou_2(H_1)$ and $Q\in\Bou_2(H_2)$ define
\begin{align*}
P\star_{w}Q
\;:=\;
(\conj_H^{-1}\,\W)^*\,(P\otimes Q)\,(\conj_H^{-1}\,\W)\; \in\; \Bou(K^{\#}).
\end{align*}
Then, for $\star = \star(T)\in \prd(\{e_i\}_{i\in I})$, we have
\begin{align*}
P\star Q = C_e^*\,(P\star_w Q)\, C_e.
\end{align*}
\item If $w=\sum_{i\in I} e_i\otimes u_i$ for $u_i\in H$, then for all $x\in H$ we define 
\begin{align*}
\phi_w(x)
:= \sum_{i\in I} {\langle x,u_i \rangle_{_{H}}} \ast \overline{e_i} \; \in \; K^{\#}.
\end{align*}
Then $\phi_w(x)$ depends on $w$ but not on its decomposition $w=\sum_{i\in I} e_i\otimes u_i$. Moreover, the vector $\rho$ in Theorem~\ref{T:thmC-for-rank1-fixed} is given by
\begin{align*}
    \rho = C_e^*\, \big( \phi_w(\mathrm{vec}(BA^\#)) \big),
\end{align*}
for all admissible $A,B$ in Theorem~\ref{T:thmC-for-rank1-fixed}.
\item Finally, the scalar ${\bf r}(\star)$ in Theorem~\ref{T:thmC-for-rank1-fixed} depends on $w$ but not on its decomposition $w=\sum_{i\in I} e_i\otimes u_i$.
\end{enumerate}
\end{theorem}

\begin{proof} 

(1) We compute:
\begin{align*}
\langle (\conj_H^{-1}\,\W)^*\,(P\otimes Q)\,(\conj_H^{-1}\,\W)\, \overline{e_j},\,\overline{e_i} \rangle_{_{K^{\#}}} 
&=\langle (P\otimes Q)\,(\conj_H^{-1}\,\W)\, \overline{e_j},\,(\conj_H^{-1}\,\W) \overline{e_i} \rangle_{_{H}}\\
=\langle (P \otimes Q )\,u_j,\,u_i \rangle_{_{H}}
&= \langle P\star Q\,e_j,\,e_i \rangle_{_{K}}\\
= \langle P\star Q\, C_e^*\,\overline{e_j},\,C_e^* \,\overline{e_i} \rangle_{_{K}}
&= \langle \,C_e \, P\star Q\, C_e^*\,\overline{e_j},\,\overline{e_i} \rangle_{_{K^{\#}}},
\end{align*}
where we used Lemma~\ref{prop:rank-one-slicing} and Lemma~\ref{lemma:canonical}. This proves (1).\medskip

\noindent(2) Suppose we have
\begin{align*}
w=\sum_{i\in I} e_i\otimes u_i=\sum_{j\in I} f_j\otimes v_j \quad \mbox{where}\quad \lVert w \rVert^2=\sum_{i\in I} \lVert u_i \rVert^2 = \sum_{j\in I} \lVert v_j \rVert^2 <\infty    
\end{align*}
for orthonormal bases $\{e_i\}_{i\in I}$ and $\{f_j\}_{j\in I}$ of $K$. Then $\{\overline{e_i}\}_{i\in I}$ and $\{\overline{f_j}\}_{j\in I}$ are orthonormal bases of $K^\#$. Suppose $\alpha_{ij}:=\langle {e_i},{f_j}\rangle_{_{K}}$. For $x\in H$, let
\begin{align*}
\phi_e:=\sum_{i\in I} \langle x,u_i \rangle_{_{H}} \ast \overline{e_i} 
\qquad \mbox{and}\qquad
\phi_f:=\sum_{j\in I} \langle x,v_j \rangle_{_{H}} \ast \overline{f_j}.
\end{align*}
We show that $\phi_e=\phi_f$, which shows the required independence of $\phi_w(x)$. 

Note that for arbitrary $h\in H$, we have $
\big\langle w, e_i \otimes {h}\big\rangle_{_{K\otimes H}} = \langle {u_i},{h}\rangle_{_{H}}$. On the other hand, we also have
\begin{align*}
\big\langle w, e_i \otimes {h}\big\rangle_{_{K\otimes H}}
&= \big \langle \sum_{j\in I} f_j \otimes v_j, e_i\otimes h \big \rangle_{_{K\otimes H}}\\
= \sum_{j\in I} \langle f_j,e_i\rangle_{_K} \big \langle v_j, h \big \rangle_{_{H}}
&= \big \langle \sum_{j\in I} \overline{\alpha_{ij}} v_j,h \big \rangle_{_{H}}.
\end{align*}
Therefore $
u_i =\sum_{j\in I} \overline{\alpha_{ij}} v_j \in H$.
Similarly $\big\langle w, f_j \otimes h\big\rangle_{_{K\otimes H}} = \langle v_j,h\rangle_{_{H}}$, and,
\begin{align*}
\big\langle w, f_j \otimes h\big\rangle_{_{K\otimes H}}
&= \big \langle \sum_{i\in I} e_i \otimes u_i, f_j\otimes h \big \rangle_{_{K\otimes H}}\\
= \sum_{i\in I} \langle e_i,f_j\rangle_{_K} \big \langle {u_i}, h \big \rangle_{_{H}}
&= \big \langle \sum_{i\in I} {\alpha_{ij}} {u_i},h \big \rangle_{_{H}}.
\end{align*}
Thus $
v_j =\sum_{i\in I} {\alpha_{ij}} u_i \in H$.
Now, the coefficients of $\phi_e$ with respect to $({\overline{f_j}})_{j\in I}$ are:
\begin{align*}
\langle \phi_e,\overline{f_j}\rangle_{_{K^{\#}}} 
&= \Big \langle \sum_{i\in I} \langle x,u_i \rangle_{_{H}} \ast \overline{e_i}, \overline{f_j} \Big \rangle_{_{K^\#}}
= \sum_{i\in I} \overline{\alpha_{ij}} \langle x,u_i \rangle_{_{H}}\\
&=  \langle x,\sum_{i\in I} {\alpha_{ij}}u_i
\rangle_{_{H}}
=  \langle x,v_j \rangle_{_{H}} = \langle \phi_f,\overline{f_j} \rangle_{_{K^{\#}}}.
\end{align*}
Thus $\phi_e=\phi_f$. Moreover, by definition $\rho=C_e^*\, \phi_w(\mathrm{vec}(BA^\#))$. \medskip

\noindent(3) By the linearity of $\mathrm{vec}^{-1}$, the above computations give
\[
\mathrm{vec}^{-1}(v_j)
=\sum_{i\in I} {\alpha_{ij}}\,\mathrm{vec}^{-1}(u_i)
\qquad (j\in I),
\]
with convergence in $\Bou_2(H_1^\#,H_2)$. Hence
$\mathcal U_f(\star):=\overline{\mathrm{span}}^{\Bou_2}\{\mathrm{vec}^{-1}(v_j):j\in I\}
\subseteq
\overline{\mathrm{span}}^{\Bou_2}\{\mathrm{vec}^{-1}(u_i):i\in I\}=:\mathcal U_e(\star)$. Using the other relation gives the reverse inclusion,
hence $\mathcal U_f(\star)=\mathcal U_e(\star)$. Consequently,
\[
{\bf r}_f(\star)
:=\sup\{\rk X:\, X\in\mathcal U_f(\star)\}
=\sup\{\rk X:\, X\in\mathcal U_e(\star)\}
=:{\bf r}_e(\star).
\]
Thus ${\bf r}(\star)$ does not depend on the decomposition $w=\sum_{i\in I}e_i \otimes u_i$.
\end{proof}

The following operator inequality is immediate.

\begin{cor}\label{cor:canonical}
With the same notation as in Theorem~\ref{T:thmC-for-rank1-fixed} and Theorem~\ref{T:canonical}, we have a canonical operator inequality, for all admissible $A,B$ in Theorem~\ref{T:thmC-for-rank1-fixed}:
\begin{align*}
(\conj_H^{-1}\,\W)^*\,(AA^*\otimes BB^*)\,(\conj_H^{-1}\,\W) \;  \succcurlyeq \; \frac{1}{\min\big(\rk AA^*,\rk BB^*,{\bf r}(\star)\big )} \ro_{\phi,\phi} \;  \succcurlyeq \; 0,
\end{align*}
where $\phi:= \phi_w(\mathrm{vec}(BA^\#)) \in K^{\#}$.
\end{cor}
\begin{proof}
    This follows from Theorem~\ref{T:thmC-for-rank1-fixed} and Theorem~\ref{T:canonical} using that $C_e$ is unitary {\cite[p. 4-5]{oertel2023beyond}} and {\cite[Proposition 16.3(e)]{conway2000course}}.
\end{proof}

As desired, the above inequality, which is unitarily equivalent to the one in Theorem~\ref{T:thmC-for-rank1-fixed}, depends only on $w\in \Hil$, and not on the decomposition $w=\sum_{i\in I} e_i\otimes u_i$ for any given orthonormal basis $\{e_i\}_{i\in I}$ of $K$. This concludes the paper.

\section*{Acknowledgements} D.G.~was partially supported by NSF grant \#2350067. P.K.V.~was supported by the Centre de recherches math\'ematiques and Universit\'e Laval (CRM-Laval) Postdoctoral Fellowship.

\bibliographystyle{plain}
\bibliography{biblio}
\end{document}